\documentclass[a4paper,11pt]{amsart}
\usepackage{amsmath}
\usepackage{amssymb}
\usepackage{amsfonts}
\usepackage{times}
\usepackage{graphicx}
\usepackage[T1]{fontenc}
\usepackage{ifthen}
\usepackage{t1enc}
\usepackage{graphics}
\usepackage{pifont}
\usepackage{float}
\usepackage{array}
\usepackage{hhline}
\usepackage{enumerate}

\def \endproof {\quad \hfill  \rule{2mm}{2mm} \par\medskip}
\DeclareMathSymbol{\subsetneqq}{\mathbin}{AMSb}{36}

\newcommand{\lec}{{\, \lesssim \, }}
\newcommand{\LL}{{\mathcal L}(\cdot)}
\newcommand{\R}{\mathbb{R}}

\newcommand{\N}{\mathbb{N}}

\newcommand{\p}{\partial}
\newcommand{\De}{\Delta}

\newcommand{\dint}{\displaystyle\int}

\def \endproof {\quad \hfill  \rule{2mm}{2mm} \par\medskip}

\newtheorem{th1}{{\bf Theorem}}[section]
\newtheorem{thm}[th1]{{\bf Theorem}}
\newtheorem{lem}[th1]{{\bf Lemma}}
\newtheorem{prop}[th1]{{\bf Proposition}}

\newtheorem{rem}[th1]{\bf Remark}

\newtheorem{defi}[th1]{\bf Definition}

\title{Strichartz type estimates and  the well posedness of an energy critical 2D wave equation in a bounded domain}
\author{S. Ibrahim}
\address{Department of Mathematics and Statistics,\\University of Victoria\\
 PO Box 3060 STN CSC\\   Victoria, BC, V8P 5C3\\ Canada}
\email{ibrahim@math.uvic.ca} \urladdr{
http://www.math.uvic.ca/~ibrahim/}
\thanks{S. I. is partially supported by NSERC\# 371637-2009 grant and a start up fund from University
of Victoria}
\thanks{R. J is partially supported by the {\sf Laboratory of PDE and
applications}, Tunisia}
\author{R. Jrad}
\email{rym.jrad@gmail.com}

\address{Laboratoire d'\'equations aux d\'eriv\'ees
partielles et applications\\Facult\'e des Sciences de Tunis,
D\'epartement de Math\'ematiques\\ Campus universitaire 1060, Tunis,
Tunisia.} \email{} \email{}
\thanks{}

\title{Strichartz type estimates and  the well posedness of an energy critical 2D wave equation in a bounded domain}

\date{\today}
\begin{document}

\begin{abstract}
We study the well-posedness of the Cauchy problem with
Dirichlet or Neumann boundary conditions associated to an
$H^1$-critical semilinear wave equation on a smooth bounded
domain $\Omega\subset\R^2$. First, we prove an appropriate Strichartz type
estimate using the $L^q$ spectral projector estimates of the Laplace
operator. Our proof follows Burq-Lebeau-Planchon \cite{BLP}. Then, we show the global well-posedness when the energy is below or at the
threshold given by the sharp Moser-Trudinger inequality. Finally, in the {\it supercritical} case, we prove an instability result using the finite speed of propagation
and a quantitative study of the associated ODE with oscillatory data.
\end{abstract}

\maketitle

\tableofcontents \vspace{ 1\baselineskip}
\renewcommand{\theequation}{\thesection.\arabic{equation}}
\newpage
\section{Introduction}
Recall the following semi-linear wave equation
\begin{equation}\label{M}
                         \begin{array}{rcl}
                         (\p_t^2-\triangle)u+f(u)=0\quad\mbox{in }\mathbb{R}_t\times{\Omega_x},\;\\\\
                         \small{u(0,x)}\small{=}\small{u_{0}(x)},\;\small{{\p}_{t}u(0,x)}\small{=}\small{u_{1}(x)},
                         \end{array}
\end{equation}
where $\Omega\subset\R^d$ ($d\geq2$) is a smooth bounded domain, $\triangle$ denotes the Laplace-Beltrami operator acting on the space variable $x$, and the nonlinearity $f$ is an odd function satisfying $f(0)=0$ and $uf(u)\geq0$. The unknown $u=u(t,x)$ is a real-valued function. Note that the above assumptions on $f$ include the massive case, namely the Klein-Gordon equation.\\
The most studied nonlinear model is when $f(u)=|u|^{p-1}u$, with $p>1$. In the case of the whole space $\Omega=\R^d$ and $d\geq3$, there is a large literature on the local and global solvability of \eqref{M} in the scale of the Sobolev spaces $H^s$ i.e. the initial data $(u_0,u_1)\in H^s\times H^{s-1}$. Among others, we refer the interested readers to, \cite{g-s-v, g-v,MG,Kap1,Kap2,MN.TO 1,MN.TO 2,Shat-Stru1,Shat-Stru2,Stru}.\\
For the global solvability in the energy space $(u_0,u_1)\in H^1\times
L^2$, there are mainly three cases. In the subscritical case
where $p<p_c=1+\frac 4 {d-2}$, Ginibre and Velo \cite{g-v} have
shown that  problem \eqref{M} has a unique solution in the space
$C(\R,H^1(\R^d))\cap C^1(\R,L^2(\R^d))$. In the critical case,
$p=p_c$, the first global well-posedness result was obtained by Struwe in the radial case
\cite{Stru}. Then Grillakis in \cite{MG} established the existence
of global smooth solutions for smooth data when $d=3$. For higher
dimensions, Shatah-Struwe \cite{Shat-Stru1,Shat-Stru2} proved the
global solvability for data in the energy space. The quintic Klein-Gordon equation in 3D was globally solved by Kapitanski \cite{Ka}. In the
supercritical case, $p>p_c$, the local well-posedness was recently
solved by Kenig-Merle \cite{KM} but for initial data in the homogeneous Sobolev spaces
${\dot{H}^{s_p}\times\dot{H}^{s_p-1}}$ with $1<s_p<3/2$. In the energy
space this is still an open problem except for some partial results about some kind of ``illposedness". See
\cite{Leb1,Leb2,chr-coll-tao} for more details. \\
If $\Omega$ is the complement of a strictly convex, smooth and
compact obstacle $\mathcal O$, problem \eqref{M} with a
Dirichlet boundary condition $u|_{\p\Omega}=0$ was
solved by Smith and Sogge for the 3D quintic equation. See \cite{SS1}. The case of a smooth
bounded domain in $\R^3$ was recently solved by Burq-Lebeau-Planchon
\cite{BLP}, and Burq-Planchon \cite{BP} who showed the existence and
uniqueness of a global solution for data in
the energy space. The major difficulty in proving such a result is to establish Strichartz type estimates. Let us recall a few historical facts about these estimates.\\
For a manifold $\Omega$ of dimension $d\geq2$ equipped with a
Riemannian metric ${g}$, Strichartz estimates are a family
of space time integral estimates on solutions : $u(t,x) :
(-T,T)\times\Omega\longrightarrow \R$ to the wave equation
\begin{equation*}
                         \begin{array}{rcl}
                         \p_t^2 u-\De_{{g}} u=0\quad\mbox{in }(-T, T)\times{\Omega_x}\;\\\\
                         \small{u(0,x)}\small{=}\small{u_{0}(x)},\;\small{{\p}_{t}u(0,x)}\small{=}\small{u_{1}(x)}.
                         \end{array}
\end{equation*}
Local Strichartz estimates state that
\begin{equation}\label{str}
\|u\|_{L^q((-T,T), L^r(\Omega))}\leq C_T
(\|u_0\|_{H^s(\Omega)}+\|u_1\|_{H^{s-1}(\Omega)}),
\end{equation}
where $H^s(\Omega)$ denotes the $L^2$-based Sobolev space, $2\leq q\leq \infty$ and $2\leq r<\infty$ satisfy
\begin{equation}\label{adm}
\frac 1 q+\frac d r =\frac d 2-s, \quad\quad\quad\frac 2 q+\frac
{d-1} r \leq\frac {d-1} 2.
\end{equation}
Estimates \eqref{str} are said global if the constant $C_T$ is $T$-independent. Estimates involving $q=\infty$ hold when $(n,q,r)\neq (3,2,\infty)$, but typically require the use of Besov spaces.\\
If $\Omega=\R^d$ and ${g}_{ij}=\delta_{ij}$, R. Strichartz
proved in \cite{str} global estimates for the wave and Shr\"odinger
equations in the diagonal case i.e. $q=r$. Then, Ginibre-Velo
\cite{g-v1} and Lindblad-Sogge \cite{LS} generalized them to the
other cases, see also Kato \cite{kat} and Cazenave-Weissler \cite{Gw}.\\
For general manifolds, phenomena such as the existence of trapped geodesics or the
finiteness of the volume can preclude the development of global
estimates,
leading us to consider just local in time estimates.\\
In the case of a compact manifold without boundary, using the finite speed of propagation and working in coordinate charts, the problem is reduced to the proof of the local Strichartz estimates for the variable coefficients wave operators on
 $\R^d$. In this context, Kapitanski in \cite{Kap3} and Mockenhaupt-Seeger-Sogge in \cite{mss} established such inequalities
 for operators with smooth coefficients. Smith in \cite{s} and Tataru in \cite{Tat} have shown Strichartz estimates
 for operators with $C^{1,1}$ coefficients. For more details, see \cite{BSS}.\\
 If $\Omega$ is a manifold with strictly geodesically-concave boundary, Smith-Sogge (see\cite{SS1}) have shown Strichartz estimates for a larger range of exponents in \eqref{adm}.\\
Using the $L^r(\Omega)$ estimates for the spectral projector
obtained by Smith-Sogge \cite{SS},
 Burq-Lebeau-Planchon established Strichartz estimates for a bounded domain for a certain range
  of triples ($q,r,s$), see \cite{BLP}. Recently, Blair-Smith-Sogge in \cite{BSS} expanded the range of indices $q$ and $r$ obtained
 in \cite {BLP} and also to other dimensions.\\
In the case where $\Omega$ is a compact convex domain in $\R^2$,
Ivanovici has very recently shown in \cite{iva} that \eqref{str} cannot hold
 when $r>4$ if $2/q+1/r=1/2$.\\

Going back to the well-posedness issues, observe that in $2D$ all nonlinearities $f$ with polynomial growths are ``subcritical" for the $H^1$ norm. This is due to the limit case of the Sobolev embedding. So, the choice of an exponential nonlinearity appears to be quite natural. Such nonlinearity was investigated by Nakamura and
Ozawa \cite{MN.TO 1,MN.TO 2}. They showed the global solvability and
established the asymptotic in time when the initial data is sufficiently small. In a
recent work, Ibrahim-Majdoub-Masmoudi \cite{sli moh nad1} considered
the case where $f(u)=ue^{4\pi u^2}$. They have quantified the size of the initial data for which one has global well-posedness. More precisely, let
$$
E_0=\|u_1\|^2_{L^2(\R^2)}+\|\nabla
u_0\|^2_{L^2(\R^2)}+\displaystyle\int_{\R^2}\frac{e^{4\pi
u_0^2}-1}{4\pi}\;dx.
$$
Then, solutions with $E_0\leq1$ exist for all time. However, in the ``supercritical" case i.e. $E_0>1$, the same authors have shown
an instability result (see \cite{sli moh nad2}), by proving the non
uniform continuity of the solution map. Recently, a similar
trichotomy was also established by
Colliander-Ibrahim-Majdoub-Masmoudi for the nonlinear
Schr\"odinger equation with the same type of nonlinearity. See \cite{CIMM}.\\

In this paper, we propose to extend the above results to the case of bounded 2D domains. We establish a trichotomy in the dynamic for both Dirichlet and Neumann type boundary conditions. More precisely, consider the 2D,  $H^1$-critical wave equation
\begin{equation}\label{C1}
                         \begin{array}{rcl}
                         (\p_t^2-\triangle_D)u+u(e^{4\pi u^2}-1)=0\quad\mbox{in }\mathbb{R}_t\times{\Omega_x}\\\\
                         \small{u(0,x)}\small{=}\small{u_{0}(x)},\;\small{{\p}_{t}u(0,x)}\small{=}\small{u_{1}(x)}\\\\
                         \small{u|_{\R_t\times\p\Omega_x}}\small{=}0,
                         \end{array}
\end{equation}
where $\Omega\subset\R^2$ is a smooth bounded domain, $u=u(t,x)$ is a real-valued function  and $\triangle_D$ denotes the Laplace-Beltrami operator with Dirichlet boundary conditions. The initial data ($u_0$,$u_1$) are in the energy space $H_0^1(\Omega)\times L^2(\Omega)$.\\
A solution $u\in {\mathcal C}([0,T],H^1)\cap{\mathcal C}^1([0,T],L^2)$ of
the Cauchy Problem \eqref{C1} satisfies the following conservation
law
\begin{equation}\label{enr}
E(u,t)=\|\p_t u(t)\|^2_{L^2(\Omega)}+\|\nabla
u(t)\|^2_{L^2(\Omega)}+\int_{\Omega}\frac{e^{4\pi
u^2}-1-4\pi u^2}{4\pi}\;dx=E(u,0).
\end{equation}
A priori, one can estimate the nonlinear part of the energy using
the following sharp Moser-Trudinger-type inequality, see for example  \cite{M}, \cite {T}.

\begin{prop}\label{MT}
Let $\Omega\subset\R^2$ be a bounded domain, and $\alpha\leq 4\pi$.
There exists a constant $C(\Omega)>0$ such that
\begin{equation}\label{MT1}
\sup_{\|\nabla u\|_{L^2(\Omega)}\leq 1} \int_{\Omega}e^{\alpha
u^2}\;dx=C(\Omega)<+\infty.
\end{equation}
Moreover, this inequality is sharp in the sense that for any $\alpha>4\pi$, the supremum in \eqref{MT1} is infinite.
\end{prop}
In our paper we take $\alpha=4\pi$, and then the discussion will be based on the
size of the initial data in the energy space. More precisely, we
distinguish the cases $E_0\leq 1$ and $E_0>1$ where $E_0=E(u,0)$ is the energy of a solution $u$.\\
Our first result is the following Strichartz type estimate\footnote{This result was stated and proved in early 2009, and now it is embedded in Theorem 1.1 in \cite{BSS}. 
However its proof is different.}.
\begin{thm}\label{st1}
Suppose that $u\in{\mathcal C}([0,T],H^1_0)\cap{\mathcal
C}^1([0,T],L^2)$ solves the linear inhomogeneous linear wave equation with Dirichlet boundary condition and $f\in L^1([0,T],L^2)$
\begin{equation}\label{cf}
                         \begin{array}{rcl}
                         (\p_t^2-\triangle_D)u=f\quad\mbox{in }\mathbb{R}_t\times{\Omega_x}\\\\
                         \small{u(0,x)}\small{=}\small{u_{0}(x)},\;\small{{\p}_{t}u(0,x)}\small{=}\small{u_{1}(x)}\\\\
                         \small{u|_{(0,T)\times\p\Omega_x}}\small{=} 0.
                         \end{array}
\end{equation}
Then, a constant $C_T$ exists such that
\begin{equation}\label{2}
 \|u\|_{L^8((0,T), C^{1/8}(\Omega))}\leq C_T (\|u_0\|_{H_0^1(\Omega)}+\|u_1\|_{L^2(\Omega)}+\|f\|_{L^1((0,T),
L^2(\Omega))}).
 \end{equation}
\end{thm}
To prove this estimate, we follow the same approach of Burq-Lebeau-Planchon \cite{BLP}  in the case of a
 bounded domain $\Omega\subset\R^3$. Their idea is based on a recent result established by Smith and Sogge \cite{SS}
to control the $L^5W^{\frac 3 {10},5}$ norm of the solution of the free wave equation by the energy norm.\\

To estimate the $L^1_T L^2_x$ norm of the nonlinear term $u(e^{4\pi u^2}-1)$,
we remark that its $L^2(\Omega)$ norm already doubles the exponent $4\pi$.
Therefore, the inequality \eqref{MT1} is insufficient to control it.
To overcome this difficulty, we use the following logarithmic
inequality with sharp constant proved in \cite{IMM}.

\begin{prop}\label{elog}
For any real number $\displaystyle\lambda>\frac{4}{\pi}$ there
exists a constant $C_\lambda$ such that, for any function
 $u$ belonging to $H_0^1(\Omega)\cap\dot{C}^{1/8}(\Omega)$, we have
\begin{eqnarray}\label{il}
\|u\|^2_{L^\infty}\leq \lambda\|\nabla
u\|^2_{L^2(\Omega)}\log(C_\lambda+\frac{\|u\|_{\dot{C}^{1/8}(\Omega)}}{\|\nabla
u\|_{L^2(\Omega)}}).
\end{eqnarray}
Moreover, the above inequality does not hold for
$\lambda=\displaystyle\frac{4}{\pi}.$
\end{prop}
Recall that for $0<\alpha<1$, $\dot{C}^{\alpha}$ denotes the homogeneous H\"older space: the set of continuous functions $u$ whose norm $\|u\|_{\dot{C}^{\alpha}}=\displaystyle\sup_{x\neq y}\frac{|u(x)-u(y)|}{|x-y|^{\alpha}}$ is finite. The inhomogeneous H\"older space is ${C}^{\alpha}=\dot{C}^{\alpha}\cap L^{\infty}$ endowed with the norm $\|u\|_{{C}^{\alpha}}=\|u\|_{\dot{C}^{\alpha}}
+\|u\|_{L^\infty}$.

Using the above propositions we can show, through a fixed point argument,  the existence of local in time solutions given by the following result.
\begin{thm}\label{loc}
Assume that ${\|\nabla u_0\|}_{L^2(\Omega)}<1$. Then, there exists a time
$T>0$ and a unique solution $u$ to problem \eqref{C1}, $u\in C([0 , T) , H_0^1(\Omega))\cap C^1([0, T) , L^2(\Omega)).$
Moreover, $u\in L^8([0,T),C^{1/8}(\Omega))$ and satisfies the energy conservation, for all $0\leq t<T$.
\end{thm}
Based on the above result and the sharp Moser-Trudinger inequality, we propose as in \cite{sli moh nad1} the following definition.

\begin{defi}
Let $E_0=E(u,t=0)$ given by \eqref{enr}.
The Cauchy problem \eqref{C1} is said to be
\begin{itemize}
  \item Subcritical if $E_0<1.$
  \item Critical if $E_0=1$
  \item Supercritical if $E_0>1.$
\end{itemize}
\end{defi}
Thanks to the energy identity \eqref{enr} and the local existence result,
we can easily show the global existence in the subcritical case as
stated in the following Theorem.
\begin{thm}\label{csc}
For any $(u_0,u_1)\in{H_0^1(\Omega)\times
L^2(\Omega)}$ with energy $E_0<1$ there is a unique global solution $u\in C(\R , H_0^1(\Omega))\cap C^1(\R , L^2(\Omega))$.
Moreover, this solution $u\in L^8_{loc}(\R,C^{1/8}(\Omega))$ and
satisfies \eqref{enr}.
\end{thm}
In the critical case we cannot apply the same arguments used in the
subcritical case. This is due to the fact that the conservation of the energy only does not rule out the possibility for the solution to (at least formally) concentrate in the sense that $$
 \limsup_{t\longrightarrow T^\star}\|\nabla u(t)\|_{L^2(\Omega)}=1.
 $$
 In such a case, we emphasize on the fact that we do not know any nonlinear estimate. Therefore, we use a multiplier techniques,
 we show that such concentration phenomena cannot occur and thus solutions are indeed global.
\begin{thm}
\label{cc} Let $(u_0,u_1)\in{H_0^1(\Omega)\times
L^2(\Omega)}$ with energy $E_0=1$. There is a unique
global solution $u$ in the space $
C(\R , H_0^1(\Omega))\cap C^1(\R , L^2(\Omega))$ with the initial data $(u_0,u_1)$.
Moreover, this solution $u\in L^8_{loc}(\R,C^{1/8}(\Omega))$ and
satisfies \eqref{enr}.
\end{thm}
In the supercritical case, we shall prove that problem
\eqref{C1} is ill-posed. Precisely, we prove
\begin{thm}\label{ins}
 Let $0<\eta<1$ be small enough. There exist a sequence of positive real number ($t_k$) tending to zero and two sequences $(v_k^\eta)$ and $(w_k^\eta)$
 of solutions of \eqref{C1} satisfying, for any $\varepsilon>0$
\begin{eqnarray}\label{ins1}
\|(v_k^\eta-w_k^\eta)(t=0,\cdot)\|^2_{H_0^1(\Omega)}+\|\p_t(v_k^\eta-w_k^\eta)(t=0,\cdot)\|^2_{L^2(\Omega)}\leq
\varepsilon\;
\end{eqnarray}
and,
\begin{eqnarray}\label{ins2}
0<E(w_k^\eta,0)-1\leq \eta^2\quad \mbox{and}\quad
0<E(v_k^\eta,0)-1\leq3\eta^2e^3
\end{eqnarray}
when $k$ is large enough. Moreover,
\begin{eqnarray}\label{ins3}
\displaystyle\liminf_{k\longrightarrow\infty}\|\p_t(v_k^\eta-w_k^\eta)(t_k,\cdot)\|^2_{L^2(\Omega)}\geq
C.
\end{eqnarray}
The constant $C$ depends only on $\eta$.
\end{thm}
 To prove this Theorem, we proceed in a similar way as in \cite{sli moh nad2}. Their idea is based on the approximation
 of the solution of the PDE by the solution of its corresponding ODE (without the ``diffusion term"). The special choice of the
 concentrating data combined to the finite speed of propagation guarantee that the two solutions indeed coincide in a
 backward light cone. Then a ``decoherence" type phenomena is shown for the ODE regime given the periodicity of its solutions.
 The local character of the proof of \cite{sli moh nad2} enables us to adapt it in our setting. This strategy was originally
  initiated by Kuksin \cite{K} and developed by Christ-Colliander-Tao \cite{chr-coll-tao}.
\begin{rem}
The results of this paper remain true if we replace the Dirichlet by
Neumann boundary conditions. This is due to the fact that we use
only $u$ and $\p_t u$ as test functions. This considerably simplifies our proof
compared to the 3D quintic problem, where in addition
$x\cdot\nabla u$ is used. That term gave rise of new boundary terms
which needed more care to control. We refer to \cite{BLP} and
\cite{BP} for full details. Also, thanks to Poincar\'e inequality, our results hold in the massive case i.e. the Klein-Gordon equation.
\end{rem}
This paper is organized as follows: in the next section, we
introduce the notation used throughout this paper. Section two is
devoted to the complete proof of our Strichartz estimates. In
section three, we combine the latter estimates with the energy
identity and the sharp logarithmic inequality to establish, through
a standard fixed point argument, the local existence results. In section
four, we focus on the proofs of Theorem \ref{csc} and
Theorem \ref{cc}. In the last section we prove the instability result given by
Theorem \ref{ins}.

\section{Notation}
For $s\geq 0$, let $H^s_D(\Omega)$ be the domain of
$(-\De_D)^{s/2}$ and $H_0^s(\Omega)$ be the closure in $H^s(\Omega)$ of the set of
smooth and compactly supported functions. Note that the space $H^s_D(\Omega)$ coincides
with $H_0^s(\Omega)$ for $0\leq s<\frac 3 2$, and that when
$s=1$, $H_0^1(\Omega)$ equipped with the inner product of
$H^1(\Omega)$ is a Hilbert space.
In this paper, the space $\displaystyle H_0^1(\Omega)$ will be endowed with the Dirichlet norm $\|u\|_{H^1_D(\Omega)}^2=\int_\Omega|\nabla u(x)|^2\;dx$.\\
It is well known that in our setting, the operator $-\De_D$ has a complete set of eigenvalues $\displaystyle\{\lambda_j^2\}_{j=0}^\infty$ and eigenfunctions. Let $m({\lambda_j})$ denote the multiplicity of $\lambda_j^2$, and $e_{\lambda_{j,k}}$ be the $k^{th}$ eigenvector in the eigenspace corresponding to the eigenvalue
$\lambda_j^2$. Then define
$\Pi_{\lambda_j}u=\displaystyle\sum_{k=1}^{m({\lambda_j})}<u,e_{\lambda_{j,k}}>e_{\lambda_{j,k}}$, where $<,>$ stands for the $L^2$ inner product. For any  $\lambda>1$, denote by $\chi_\lambda$ the spectral projection given by
\begin{equation*}
\chi_{\lambda}u=\sum_{\begin{array}{c}\scriptstyle
\{j\;/\;\lambda\leq\lambda_j<\lambda+1\}
\end{array}}\Pi_{\lambda_j}u.
\end{equation*}
Finally, let ${{|D|}}:=\sqrt{-\De_D}$. For any $0<S<T$ and $x_0\in\Omega$, define :

$$
K^T_S(x_0)=\{(x,t)/ |x-x_0|<t\;,\;S<t<T\;,\; x\in \Omega\},\quad \mbox{the backward light cone}
$$

\begin{eqnarray}\label{mst}
M_S^T(x_0)=\{(x,t)/ |x-x_0|=t,\;x\in\Omega,\;S<t<T\}\quad \mbox{its mantle}
\end{eqnarray}
and for fixed $t$
$$
 D_t(x_0)=\{x/ |x-x_0|<t\}\cap \Omega \quad \mbox{its space like sections}.
$$
Observe that

$$
\p K_S^T(x_0)=(([S,T]\times\p\Omega)\cap K^T_S(x_0))\cup D_S\cup D_T\cup
M_S^T(x_0).
$$
Finally, let $e(u)$ be the energy density
\begin{eqnarray}\label{e}
e(u)=(\p_tu)^2+|\nabla u|^2+\frac{e^{4\pi u^2}-1-4\pi u^2}{4\pi}.
\end{eqnarray}
When $x_0=0$, we remove the dependence upon $x_0$ in the above notation. Let $\mathcal E_T$ be the space defined as follows

\begin{equation}\label{epsi}
\mathcal E_T=C([0,T],H_0^1(\Omega))\cap C^1([0,T],L^2(\Omega))\cap
L^8([0,T],C^{1/8}(\Omega)).
\end{equation}
Set
\begin{equation}\label{NT}
\| u\|_T=\displaystyle\sup_{ t\in[0,T]}
 (\| u(t,\cdot)\|_{H_0^1(\Omega)}+\|{\p_{t}u(t,\cdot)}\|_{L^{2}(\Omega)})+
 \| u\|_{L^{8}([0,T],\;C^{1/8}(\Omega))}.
\end{equation}
Recall that $\mathcal E_T$ equipped with the norm $\|\cdot\|_T$ is a
complete space.

\section{Strichartz estimate}
In this section, we prove our appropriate Strichartz estimate
given by Theorem \ref{st1}. The proof follows Burq-Lebeau-Planchon \cite{BLP}. It is based on an estimate in Lebesgue spaces of the spectral projector
$\chi_{\lambda}$. This estimate is due to Smith-Sogge \cite{SS}. First, we recall this estimate in two space dimensions.

\begin{prop}\label{SS1}[Smith-Sogge \cite{SS}]
Let $\Omega\subset\R^2$ be a smooth bounded domain. Then the
following estimate
\begin{equation}\label{sog1}
\|\chi_\lambda u\|_{L^q(\Omega)}\leq C \lambda^{\frac 2 3
(\frac{1}{2}-\frac{1}{q})}\|u\|_{L^2(\Omega)}
\end{equation}
holds for $2\leq q\leq 8$.
\end{prop}

\emph{Proof of Theorem \ref{st1}}. In this proof, we distinguish two cases.\\

\textbf{First case : Estimate for the homogeneous problem i.e.  when $f=0$.}\\
In this case, Duhamel's formula gives
\begin{eqnarray*}
u = \cos(t{|D|})u_0+\frac{\sin(t{|D|})}{{|D|}}u_1
\end{eqnarray*}
where
\begin{eqnarray*}
\cos(t{|D|}):=\left(\frac{e^{it
{|D|}}+e^{-it{|D|}}}{2}\right),\quad
{\sin(t{|D|})}:=\left(\frac{e^{it{|D|}}-e^{-it{|D|}}}
{2i}\right)
\end{eqnarray*}
and ${\mathcal L}(t)u_0:=e^{\pm it{|D|}}u_0$ is the solution $u$ of $\partial_tu=\pm i{|D|}u$ and $u(t=0)=u_0$. By Minkowski inequality
$$
\|u\|_{L^8((0,1), C^{1/8}(\Omega))}\leq \|\LL u_0\|_{L^8((0,1),
C^{1/8}(\Omega))}+\|\LL(\frac{u_1}{{|D|}})\|_{L^8((0,1)
C^{1/8}(\Omega))}.
$$
Therefore, we need to estimate $\|\LL u_0\|$ in $L^8((0,1),
C^{1/8}(\Omega))$. \\
\textbf{Step 1 : } We show that
\begin{equation}\label{a1}
\|e^{itA}u_0\|_{L^8((0,2\pi), L^8(\Omega))}\leq C\|u_0\|_{H_D^{\frac
5 8}(\Omega)},
\end{equation}
where $A$ is the ``modified" ${|D|}$ operator with integer
eigenvalues i.e.
$$
A(e_\lambda)=[\lambda]e_\lambda.
$$
The notation $[\cdot]$ stands for the integer part and
$e_{\lambda}$ is an eigenfunction of $-\triangle_D$ associated to the
eigenvalue $\lambda^2$ (Hence an eigenfunction of ${|D|}$ associated to the eigenvalue $\lambda$).\\
\\
Since $u_0$ is in $L^2(\Omega)$, we can write
$$
u_0(x)=\sum_{\lambda\in \sigma(\sqrt{-\triangle_D})}<u_0,
e_\lambda>e_\lambda(x)=:\sum_{\lambda\in
\sigma(\sqrt{-\triangle_D})}u_\lambda e_\lambda(x),
$$
where $\sigma(\sqrt{-\triangle_D})$ denotes the spectrum of $\sqrt{-\triangle_D}$.\\
So,
\begin{eqnarray*}
e^{itA}u_0(x)&=&\sum_{\lambda\in \sigma(\sqrt{-\triangle_D})}e^{itA}u_\lambda e_\lambda(x)\\
&=&\sum_{\lambda\in \sigma(\sqrt{-\triangle_D})}e^{it[\lambda]}u_\lambda e_\lambda(x).\\
\end{eqnarray*}
Setting $k=[\lambda]$, we have
\begin{eqnarray*}
e^{itA}u_0(x)=\sum_{k=1}^{\infty}e^{itk}C_k(x),
\end{eqnarray*}
\\
where the Fourier coefficient $C_k(x)$ is given by
$$C_k(x)=\displaystyle\sum_{\begin{array}{c}\scriptstyle\lambda\in
\sigma
(\sqrt{-\triangle_D})\\
\scriptstyle k\leq \lambda<k+1
\end{array}}u_\lambda e_\lambda(x).$$
Thanks to the 1D Sobolev embedding, $H^{\frac{1}{2}-\frac 1
q}(0,2\pi)\hookrightarrow L^q(0,2\pi)$ for all $q\geq 2$, we have
\begin{eqnarray*}
\|e^{itA}u_0\|_{L^8(\Omega,L^8(0,2\pi))}^2
&\lec& (\int_\Omega \|e^{itA}u_0(x)\|_{H^{\frac 3 8}(0,2\pi)}^8\;dx)^{1/4}\\
&\lec& \left\|\;\|e^{itA}u_0(x)\|_{H^{\frac 3
8}(0,2\pi)}^2\right\|_{L^{4}(\Omega)}.
\end{eqnarray*}
Then, Parseval's formula gives
\begin{eqnarray*}
\|e^{itA}u_0(x)\|_{H^{\frac 3 8}(0,2\pi)}^2&=&\sum_{k\geq1}(1+k)^{3/4}\|e^{itk}C_k(x)\|_{L^2(0,2\pi)}^2\\
&\lec&  \sum_{k\geq1}(1+k)^{3/4}|C_k(x)|^2.
\end{eqnarray*}
Now applying Minkowski inequality and using estimate \eqref{sog1},
we obtain
\begin{eqnarray*}
\|e^{itA}u_0\|_{L^8(\Omega,L^8(0,2\pi))}^2&\lec& \left\|\sum_{k\geq1}(1+k)^{3/4}
|C_k(x)|^2\right\|_{L^{4}(\Omega)}\\
&\lec& \sum_{k\geq1}(1+k)^{3/4}\|C_k\|_{L^8(\Omega)}^2=C\sum_{k\geq1}(1+k)^{3/4}\|\chi_k u_0\|_{L^8(\Omega)}^2\\
&\lec& \sum_{k\geq1}(1+k)^{\frac 3 4}k^{\frac 1 2}\|\chi_k u_0\|_{L^2(\Omega)}^2\\
&\lec& \sum_{k\geq1}(1+k)^{\frac 5
4}\sum_{\begin{array}{c}\scriptstyle\lambda\in \sigma
(\sqrt{-\triangle_D})\\
\scriptstyle k\leq \lambda<k+1
\end{array}}|u_{\lambda}|^2\sim \|u_0\|_{H_D^{\frac 5 8}(\Omega)}^2,
\end{eqnarray*}
which gives
\begin{eqnarray*}
\|e^{itA}u_0\|_{L^8(\Omega,L^8(0,2\pi))}^2\lec \|u_0\|_{H_D^{\frac 5
8}(\Omega)}^2
\end{eqnarray*}
as desired.

\textbf{Step 2 :} We prove \eqref{a1} for the operator $\LL$.
\begin{equation}\label{a2}
\|\LL u_0\|_{L^8((0,2\pi), L^8(\Omega))}\leq C\|u_0\|_{H_D^{\frac 5
8}(\Omega)}.
\end{equation}

Let $v=e^{it{|D|}}u_0$. It is clear that $v$ satisfies
\begin{equation*}\left\{
                                                         \begin{array}{ll}
                                                           (\p_t-iA)v=(-iA+i{|D|})v \\
                                                           v|_{t=0}=u_0,
                                                         \end{array}
                                                       \right.
                                                       \end{equation*}
and according to Duhamel's formula
$$
v(t,x)=e^{itA}u_0(x)+\int_0^te^{i(t-s)A}(-iA+i{|D|})v(s,x)\;ds.
$$
So, using H\"older in the second estimate
\begin{eqnarray*}
\|v(\cdot,x)\|_{L^8(0,2\pi)}&\leq& \|e^{itA}u_0(x)\|_{L^8(0,2\pi)}+
\|\int_0^te^{i(t-s)A}(-iA+i{|D|})v(s,x)\;ds\|_{L^8(0,2\pi)}\\
&\lec& \|e^{itA}u_0(x)\|_{L^8(0,2\pi)}+
(\int_0^{2\pi}\int_0^{2\pi}|e^{i(t-s)A}(-iA+i{|D|})v(s,x)|^8\;ds\;dt)^{1/8}\\
&\lec& (\|e^{itA}u_0(x)\|_{L^8(0,2\pi)}+
(\int_0^{2\pi}\|e^{i(t-s)A}(-iA+i{|D|})v(s,x)\|_{L^8(0,2\pi)}^8\;ds)^{1/8}.
\end{eqnarray*}
Applying \eqref{a1}
\begin{eqnarray*}
\|v\|_{L^8(\Omega, L^8(0,2\pi))}&\lec& \left(\|u_0\|_{H_D^{\frac 5
8}(\Omega)}+(\int_{\Omega}\int_0^{2\pi}\|e^{i(t-s)A}(-iA+i{|D|})v(s,x)\|_{L^8(0,2\pi)}^8\;ds\;dx)
^{1/8}\right)\\
&\lec& \left(\|u_0\|_{H_D^{\frac 5
8}(\Omega)}+(\int_0^{2\pi}\|e^{i(t-s)A}(-iA+i{|D|})v(s)\|_{L^8(\Omega,L^8(0,2\pi))}^8\;ds)
^{1/8}\right).\\
\end{eqnarray*}
 Since the operator $A-{|D|}$ is bounded on $H_D^{\frac 5 8}$, then
\begin{eqnarray*}
\|v\|_{L^8(\Omega, L^8(0,2\pi))}&\lec& \|u_0\|_{H^{\frac 5
8}(\Omega)}
+(\int_0^{2\pi}\|(A-{|D|})v(s)\|^8_{H_D^{\frac 5 8 }(\Omega)}\;ds)^{1/8}\\
&\lec& \|u_0\|_{H_D^{\frac 5 8 }(\Omega)}+ (\int_0^{2\pi}\|v(s)\|^8_{H_D^{\frac 5 8 }(\Omega)}\;ds)^{1/8}\\
&\lec& \|u_0\|_{H_D^{\frac 5 8}}+ C(\int_0^{2\pi}\|u_0\|^8_{H_D^{\frac 5 8 }(\Omega)}\;ds)^{1/8}\\
&\lec&\|u_0\|_{H_D^{\frac 5 8}(\Omega)},
\end{eqnarray*}
where we used $\displaystyle\sup_{t\in[0,2\pi]}\|v(t)\|_{H_D^{\frac 5 8 }(\Omega)}=
\|e^{it{|D|}}u_0\|_{H_D^{\frac 5 8 }(\Omega)}\leq C\|u_0\|_{H_D^{\frac 5 8 }(\Omega)}$ in the last inequality.\\
As a consequence, we obtain \eqref{a2} as desired.

\textbf{Step 3 : } We show that
$$
\|u\|_{L^8((0,1), C^{1/8}(\Omega))}\lec
(\|u_0\|_{H^1_D(\Omega)}+\|u_1\|_{L^2(\Omega)}).
$$
Recall the following elliptic regularity result:
\begin{eqnarray*}
-\triangle_D u+u=g \in L^q(\Omega)\,\mbox{and
}u\mid_{\p\Omega}=0\Rightarrow u\in W^{2, q}(\Omega)\cap
W_0^{1,q}(\Omega)\;\mbox{and } \|u\|_{W^{2,q}(\Omega)}\lec
\|g\|_{L^q(\Omega)}.
\end{eqnarray*}
Assuming $u_0\in{\mathcal C}_0^\infty(\Omega)$, we have for almost
all $t$
$$-\triangle_D{\mathcal L}(t)u_0+{\mathcal L}(t)u_0={\mathcal L}(t)(-\triangle_D u_0+u_0)\in L^8(\Omega).
$$
Thus
\begin{eqnarray*}
\|{\mathcal L}(t) u_0\|_{W^{2,8}(\Omega)}&\lec& \|-\triangle_D({\mathcal L}(t) u_0)+{\mathcal L}(t) u_0\|_{L^8(\Omega)}\\
&\lec&
(\|{\mathcal L}(t) (\triangle_D u_0)\|_{L^8(\Omega)}+\|{\mathcal L}(t) u_0\|_{L^8(\Omega)}),\\
\end{eqnarray*}
and therefore
\begin{eqnarray*}
\|\LL u_0\|_{L^8((0,1),W^{2,8}(\Omega)} \lec(\|\LL (\triangle_D
u_0)\|_{L^8((0,1),L^8(\Omega))}+\|\LL
u_0\|_{L^8((0,1),L^8(\Omega))}).
\end{eqnarray*}
Applying \eqref{a2} to $\triangle_D u_0$ we obtain
\begin{eqnarray*}
\|\LL (\triangle_D u_0)\|_{L^8((0,1), L^8(\Omega))}\leq \|\LL
(\triangle_D u_0)\|_{L^8((0,2\pi), L^8(\Omega))} \lec \|\triangle_D
u_0\|_{H_D^{\frac 5 8}(\Omega)}\lec \|u_0\|_{H_D^{\frac {21}
8}(\Omega)}.
\end{eqnarray*}
Consequently
\begin{eqnarray}\label{in2}
\|\LL u_0\|_{L^8((0,1),W^{2,8}(\Omega)} \lec \|u_0\|_{H_D^{\frac
{21} 8}(\Omega)}.
\end{eqnarray}
Applying the complex interpolation to \eqref{a2} and \eqref{in2}
with $\theta=\frac{3}{16}$, we have

\begin{eqnarray}\label{in3}
\|\LL u_0\|_{L^8((0,1),W^{\frac 3 8,8}(\Omega))}\lec
\|u_0\|_{H_D^1(\Omega)}.
\end{eqnarray}
Now, by Sobolev embedding, we have for all $p\leq 8$
$$
W^{\frac1 8+\frac 2 p, 8}(\Omega)\hookrightarrow C^{\frac 2 p-\frac 1
8}(\Omega)\hookrightarrow C^{\frac 1 8}(\Omega).
$$
Thus, we can rewrite \eqref{in3} as
\begin{eqnarray}\label{con}
\|\LL u_0\|_{L^8((0,1), C^{1/8}(\Omega))}\lec
\|u_0\|_{H^1_D(\Omega)}
\end{eqnarray}
which implies that
\begin{eqnarray*}
\|u\|_{L^8(0,1), C^{1/8}(\Omega))}\lec
(\|u_0\|_{H^1_D(\Omega)}+\|\frac{u_1}{{|D|}}\|_{H^1_D(\Omega)}).
\end{eqnarray*}
Finally, we use the fact that ${{|D|}}^{-1}$ is an isometry from
$L^2(\Omega)$ to $H_0^1(\Omega)$ to conclude that
$$
\|u\|_{L^8((0,1), C^{1/8}(\Omega))}\lec
(\|u_0\|_{H^1_D(\Omega)}+\|u_1\|_{L^2(\Omega)}).
$$

\textbf{Second case: An arbitrary $f\in L^1(L^2).$}\\

Thanks to Duhamel's formula
$$
u(t,x)=\cos(t{|D|})u_0+\sin(t{|D|}){{|D|}}^{-1}u_1
+\int_0^t\sin{((t-s){|D|})}{{|D|}}^{-1}f(s)\;ds.
$$
Applying the result of the first case, we obtain
\begin{eqnarray*}
\|u\|_{L^8((0,1), C^{1/8}(\Omega))} &\lec &
\left(\|u_0\|_{H_0^1(\Omega)}+\|u_1\|_{L^2(\Omega)}\right)\\
&+&\int_0^t\|\sin{((t-s){|D|})}{{|D|}}^{-1}f(s)\|_{L^8((0,1), C^{1/8}(\Omega))}\;ds\\
&\lec&
\left(\|u_0\|_{H_0^1(\Omega)}+\|u_1\|_{L^2(\Omega)}\right)+\int_0^1 \|{{|D|}}^{-1}f(s)\|_{H^1_0(\Omega)}\;ds\quad(\mbox{by }\eqref{con})\\
&\lec&
  \left(\|u_0\|_{H_0^1(\Omega)}+\|u_1\|_{L^2(\Omega)}+\int_0^1 \|f(s)\|_{L^2(\Omega)}\;ds\right)\\
&\lec&
\left(\|u_0\|_{H_0^1(\Omega)}+\|u_1\|_{L^2(\Omega)}+\|f\|_{L^1((0,1),L^2(\Omega))}\right).
\end{eqnarray*}
This completes the proof of Theorem \ref{st1}.
\endproof

Now we show that in the supercritical case, it is impossible to
estimate the nonlinear term in any dual Strichartz norm. Our result
stands for the solutions to the free  wave equation which is the first iteration in any iterative scheme for
the nonlinear problem. We emphasize on the fact that the linear energy is less than one, and the nonlinear one is slightly bigger than one (supercritical). More precisely we have

\begin{prop}
\label{supercriticalST}For any $\delta>0,$ there exists a sequence $(v_k)$ of solutions
of the free wave equation such that we have
\begin{equation}
\label{energysize} \|\nabla
v_k(0)\|_{L^2}^2+\|v_k(0)\|_{L^2}^2+\|{\partial_t
v}_k(0)\|_{L^2}^2<1,\quad\mbox{and}\quad    E(v_k, t=0)\leq 1+\delta
\end{equation}
for $k$ large, while for any $T>0$, any $p,q\in[1,\infty]$ satisfying $\frac 1 p+
\frac 2 q\leq2$
\begin{equation}
\|f(v_k)\|_{L^p((0,T),L^q(\Omega))}\geq C_{\delta}\sqrt k.
\end{equation}
\end{prop}

\begin{proof}
Without loss of generality, we may assume that $0\in\Omega$. Let $\delta>0$, and choose $p$, $q$ such that
$$
\frac 1 p+\frac 2 q\leq2.
$$
For any
$k\geq 1$, let $v_k$ be the solution of the free wave equation with data
\begin{eqnarray*}
v_k(0,x)=(1-\frac{2a}{k})f_k(ax)\quad\mbox{and}\quad
\partial_t{v_k}(0,x)=0,
\end{eqnarray*}
where $a>1$ to be chosen in the sequel. The functions $f_k$ are defined by
\begin{eqnarray}
\label{lesfk}
f_k(x)=\left\{
  \begin{array}{ll}

    0\quad\quad\quad \mbox{if}\quad|x|\geq 1\\
\\
  \displaystyle -\frac{\log|x|}{\sqrt{k\pi}}\quad\quad\quad \mbox{if}\quad  e^{-k/2}\leq|x|\leq 1  \\
\\

   \displaystyle \sqrt{\frac{k}{4\pi}}\quad\quad\quad \mbox{if}\quad |x|\leq e^{-k/2},
  \end{array}
\right.
\end{eqnarray}
and were introduced in \cite{M} to show the optimality of the
exponent $4\pi$ in Trudinger-Moser inequality.\\ Let $a_1>1$ be sufficiently large such that the ball
$B(0,1/a_1)\subset\Omega$. For $a>a_1$ we have,
\begin{eqnarray*}
\|\nabla v_k(0)\|^2_{L^2(\Omega)}=\frac 2
k(1-\frac{2a}{k})^2\int_{\frac{1}{a}e^{-k/2}}^{\frac 1
a}\frac{dr}{r}= (1-\frac{2a}{k})^2
\end{eqnarray*}
and
\begin{eqnarray*}
\| v_k(0)\|^2_{L^2(\Omega)}&=&(1-\frac{2a}{k})^2\left[\frac 2
{ka^2}\int_{e^{-k/2}}^{1} r\log ^2 r\;dr+\frac k 2\int_{0}^{\frac 1
a e^{-k/2}} r\;dr\right]\\
&\leq& \frac{1}{2ka^2}(1-\frac{2a}{k})^2.
\end{eqnarray*}
As  $v_k(0,x)$ can be extended (by zero outside its support) as an $H^1(\R^2)$, then the Trudinger-Moser inequality:
$$ \label{TM}
  \|\nabla\varphi\|_{L^2(\R^2)}< {1} \implies \int_{\R^2} (e^{4\pi|\varphi|^2}-1) dx \lec \frac{\|\varphi\|_{L^2(\R^2)}^2}{1-\|\nabla\varphi\|_{L^2(\R^2)}^2}
$$
shows that
$$
\int_{\Omega}(e^{4\pi v^2_k(0,x)}-1)\;dx\leq \int_{\R^2}(e^{4\pi v_k(0,x)^2}-1)\;dx\leq \frac{C}{a^3},
$$
for an absolute constant $C$. Therefore, we can choose $a_2>a_1$ such that $\frac{C}{a_2^3}\leq\delta$. Thus,  for $a\geq a_2$ and $k$ large enough we have
$$
\| v_k(0)\|^2_{L^2(\Omega)}+\|\nabla
v_k(0)\|^2_{L^2(\Omega)}+\int_{\Omega}(e^{4\pi v^2_k(0)}-1)\;dx\leq
1+\delta,
$$
and  \eqref{energysize}  follows. Next, by the finite speed of propagation, we know that
$$
v_k(t,x)=(1-\frac {2a}{k})\sqrt{\frac{k}{4\pi}}
$$
for any $(t,x)$ in the backward light cone
$$
K_0^k:=\{(x,t):\;0\leq t\leq\frac{e^{-k/2}}{a}\quad\mbox{and}\quad
|x|\leq \frac{ e^{\frac{-k}{2}}}{a}-t\}.
$$
Thus for $k$ large enough (eventually with respect to $a$)

\begin{eqnarray*}
v_k (e^{4\pi v_k^2}-1)&\geq& (1-\frac
{2a}{k})\sqrt{\frac{k}{4\pi}}\big(\exp{\Big((1-\frac{2a}{k})^2k\Big)}-1\big)\\
&\geq& C \sqrt ke^k.
\end{eqnarray*}
Now choosing $k$ larger  so that $e^{-k/2}\leq T$, we have the estimate
\begin{eqnarray*}
\|v_k (e^{4\pi v_k^2}-1)\|_{L^p((0,T),L^q(\Omega))}&\geq& \|v_k( e^{4\pi
v_k^2}-1)\|_{L^p((0,\frac 1 a e^{\frac{-k}{2}}),L^q(|x|\leq \frac 1 a
e^{\frac{-k}{2}}-t))}\\
&\geq& C\sqrt{k}e^k\left(\frac{e^{-k/2}}{a}\right)^{\frac 2 q+\frac 1 p}\\
&\geq& C\frac{\sqrt{k}}{a^2},\\
\end{eqnarray*}
where we used the fact that $\frac 2 q+\frac 1 p\leq2$ in the third inequality.
\end{proof}

\section{The local existence}
In this section, we prove Theorem \ref{loc}. We start by giving two
Lemmas. The first one provides the nonlinear estimate needed for the
fixed point argument. The second one will be used to show the
unconditional uniqueness result\footnote{Uniqueness in the
energy space and not in just a subspace of it.}. The $\R^2$-counter parts of these Lemmas can be found
in \cite{sli moh nad1}.
\begin{lem}
\label{lem1}
 Fix a time $T>0$ and $0<A<1$, and denote by $f(u)=u(e^{4\pi u^2}-1)$. There exists $0<\gamma=\gamma(A)<8$ such that if
 $$
 u_1 \;,\; u_2\;  \in C([0 , T] ,
H_0^1(\Omega))\cap C^1([0 , T] , L^2(\Omega))\cap
L^8([0,T],C^{1/8}(\Omega))
$$

and
 $$
 \displaystyle\sup_{ t\in[0,T]}
 \Vert\nabla u_i(t,\cdot)\Vert_{L^{2}({\Omega})}\leq
 A \quad\quad i=1,2,
 $$
then \\

 $$
 {\| f(u_1)-f(u_2)\|}_{{L^1_T}(L^{2}(\Omega))}\leq c{\| u_1-u_2\|}_{T}\left(
 T+T^{1-\frac{\gamma}{8}}\left(\left(\displaystyle\frac{{\Vert u_1\Vert}_T}{A}\right)^{\gamma}+\left(\displaystyle\frac{{\Vert u_2\Vert}_T}{A}\right)^{\gamma}\right)\right),
 $$
 where the norm $\|\cdot\|_T$ is defined by \eqref{NT}.
\end{lem}

\begin{proof} Thanks to the mean value theorem we can write
$$
u_1(e^{4\pi
u_1^2}-1)-u_2(e^{4\pi u_2^2}-1)=u[(1+8\pi u_{\theta}^2)e^{4\pi
u_{\theta}^2}-1]
$$
with $u=u_1-u_2$ and $0\leq\theta\leq1$.  Set $u_{\theta}=(1-\theta)u_1+\theta u_2$. We have
$$
\sup_{t\in[0,T]}\Vert |D|
u_{\theta}(t,\cdot)\Vert_{L^{2}({\Omega})}\leq
 A.
$$
So
\begin{eqnarray*}
\left\|f(u_1)-f(u_2)\right\|_{L^1_T(L^2(\Omega))}
=\left\|u[(1+8\pi u_{\theta}^2)e^{4\pi
u_{\theta}^2}-1]\right\|_{L^1_T(L^2(\Omega))}.
\end{eqnarray*}
On the other hand, observe that for any $a>0$ and $\varepsilon>0$,
\begin{eqnarray}\label{r}
\;(1+2a)e^{a}-1\leq
2(1+\frac{1}{\varepsilon})(e^{(1+\varepsilon)a}-1).
\end{eqnarray}
Then, H\"older inequality together with Sobolev embedding and the
above observation yield

\begin{eqnarray*}
\left\| u[(1+8\pi u_{\theta}^2)e^{ 4\pi
u_{\theta}^2}-1]\right\|^2_{L^2(\Omega)}&\leq&
C_{\varepsilon}\left\| u(e^{4\pi(1+\varepsilon)
u_{\theta}^2}-1)\right\|^2_{L^2(\Omega)}\\
&\leq& C_{\varepsilon} \left\|u(t)\right\|^2_{L^{2+\frac 2
\zeta}(\Omega)}
\left\| (e^{4\pi(1+\varepsilon)u_{\theta}^2}-1)^2\right\|_{L^{1+\zeta}(\Omega)}\\
&\leq&C_{\varepsilon}
\left\|u(t)\right\|^2_{H^1(\Omega)}e^{4\pi(1+\varepsilon)\|u_{\theta}\|^2_{L^{\infty}(\Omega)}}\left\|
 e^{4\pi(1+\varepsilon)u_{\theta}^2}-1\right\|_{L^{1+\zeta}(\Omega)},
\end{eqnarray*}
for any $\varepsilon>0$. Moreover, since
${\|u_{\theta}\|}^2_{H^1_0(\Omega)}\leq {A}^2$, then the
Moser-Trudinger inequality \eqref{MT1} implies that
\begin{eqnarray*}
\int_{\Omega}
  (e^{4\pi(1+\varepsilon)u_{\theta}^2}-1)^{1+\zeta}\;dx\leq \int_{\Omega}
  (e^{4\pi(1+\varepsilon)(1+\zeta)
 u_{\theta}^2}-1)\;dx\leq C{(\Omega,A)},
 \end{eqnarray*}
provided that $\varepsilon>0$ and $\zeta>0$ are chosen such that
$(1+\varepsilon)(1+\zeta){A}^2<1$.\\
Thanks to the log estimate \eqref{il} for $\lambda>4/{\pi}$ there
is a constant $C_\lambda>1$ such that
$$
e^{{4\pi(1+\varepsilon){\Vert
u_{\theta}(t)\Vert}^2_{L^\infty(\Omega)}}}\leq\exp\left(4\pi\lambda(1+\varepsilon){\Vert
|D|
u_{\theta}(t)\Vert}^2_{L^2(\Omega)}\log\left(C_{\lambda}+\frac{\|u_{\theta}\|_{C^{1/8}(\Omega)}}{\||D|
u_{\theta}(t)\|_{L^2(\Omega)}}\right)\right).
$$
Using the fact that for any $B_1>1$, $B_2>0$, the function $x\longmapsto
x^2\log(B_1+\displaystyle\frac{ B_2}{x})$ is non-decreasing, we
deduce that

\begin{eqnarray*}
e^{{4\pi(1+\varepsilon){\|u_{\theta}(t)\|}^2_{L^\infty(\Omega)}}}&\leq&\exp\left(4\pi(1+\varepsilon)\lambda
{A}^2 \log(C_\lambda+\frac{{\|u_{\theta}(t)\|}_{C^{1/8}(\Omega)}}{A})\right)\\
&\leq&\left(C_\lambda+\frac{{\|u_{\theta}(t)\|}_{C^{1/8}(\Omega)}}{A}\right)^{4\pi(1+\varepsilon)\lambda
{A}^2}.
\end{eqnarray*}
Setting $\gamma=2\pi\lambda(1+\varepsilon){A}^2$, we have

$$
 \left\| u[(1+8\pi u_{\theta}^2)e^{ 4\pi
u_{\theta}^2}-1]\right\|_{L^2(\Omega)} \leq C_{(A,\Omega)}{\|
u(t)\|}_{H^1(\Omega)}\left(C_\lambda+\frac{{\|u_{\theta}(t)\|}_{C^{1/8}}}{A}\right)^{\gamma}.
$$
Now since $A<1$, we can choose $\lambda$ such that $0<\gamma<8$. Thus
\begin{eqnarray*}
\int_0^{T}\left\| u[(1+8\pi u_{\theta}^2)e^{ 4\pi
u_{\theta}^2}-1]\right\|_{L^2(\Omega)}\;dt \leq
C_{(\Omega,A)}{\|u\|}_{T}\int_0^{T}\left(C+\frac{{\|
u_{\theta}(t)\|}_{C^{1/8}(\Omega)}}{A}\right)^{\gamma}\;dt
\end{eqnarray*}

\begin{eqnarray*}
&\leq&
 C_{(\Omega,A)} {\left\| u\right\|}_{T}
\left[T+\left\|\left(\frac{\|u_1(t)\|_{C^{1/8}(\Omega)}}{A}\right)^{\gamma}\right\|_{L^1_T}
+\left\|\left(\frac{\|u_2(t)\|_{C^{1/8}(\Omega)}}{A}\right)^{\gamma}\right\|_{L^1_T}\right]\\
&\leq&
 C_{(\Omega,A)}
{\|u\|}_{T}\left[T+T^{\frac{8-\gamma}{8}}\left(\left(\frac{{\|
u_1\|}_T}{A}\right)^{\gamma}+\left(\frac{{\|
u_2\|}_T}{A}\right)^{\gamma}\right)\right].
\end{eqnarray*}
Finally, we obtain
$$
{\|f(u_1)-f(u_2)\|}_{L^1_T(L^2(\Omega))}\leq C_{(\Omega,A)} {\|
u\|}_{T}\left[T+T^{1-\frac{\gamma}{8}}\left(\left(\frac{{\|
u_1\|}_T}{A}\right)^{\gamma}+\left(\frac{{\|
u_2\|}_T}{A}\right)^{\gamma}\right)\right]
$$
as desired.
\end{proof}
\begin{lem}\label{lem2}
Let $F(u)=e^{4\pi u^2}-1$ and $u\in C([0 , T] , H_0^1(\Omega))\cap
C^1([0 , T] , L^2(\Omega))$ be the solution of \eqref{C1} with $u(t=0)=u_0$ such that $\|\nabla
u_0\|_{L^2(\Omega)}<1$. Then there exists a continuous real valued
function $C(t)$, vanishing at zero such that
$$
\|F(u)\|_{L^1_T(L^2(\Omega))}\leq C(T).
$$
\end{lem}
\emph{Proof}. Write the solution $u$ of  \eqref{C1} as $u=v_L+\widetilde{u}$ where $v_L$ solves the free wave equation with the same data as $u$ and $\widetilde{u}$
belongs to $C([0 , T] , H_0^1(\Omega))\cap C^1([0 , T] ,
L^2(\Omega))$ and solves the perturbed problem

\begin{equation} \left\{
\begin{array}{llll}

\Box\widetilde{u}=-(v_L+\widetilde{u})(e^{4\pi
(v_L+\widetilde{u})^2}-1)\\
\widetilde{u}(0,x)=0\\
{\p}_{t}\widetilde{u}(0,x)=0.
\end{array}
\right.
\end{equation}
Recall the following trivial observations
\begin{equation}\label{ob1}
\mbox{For\;all}\;\varepsilon>0,\quad(\widetilde{u}+v_L)^2\leq
(1+\frac{1}{2\varepsilon}) \widetilde{u}^2+(1+2\varepsilon)v_L^2
\end{equation}

\begin{equation}\label{ob2}
e^{x+y}-1=(e^{x}-1)(e^{y}-1)+(e^{x}-1)+(e^{y}-1)
\end{equation}
and
\begin{equation}\label{ob3}
\mbox{for\;all}\;x\geq 0,\;\mbox{and }
 \;\alpha>1,\quad (e^x-1)^{\alpha}\leq e^{\alpha x}-1.
 \end{equation}
Set $a=(1+2\varepsilon)$ and $b= (1+\frac 1 {2\varepsilon})$, then observations \eqref{ob1} and \eqref{ob2} imply

\begin{eqnarray*}
\left(e^{4\pi (v_L+\widetilde{u})^2}-1\right)^2&\leq&\left(e^{4\pi(a
v_L^2+b \widetilde{u}^2)}-1\right)^2\\
&\leq&3\left[(e^{4\pi a v_L^2}-1)^2(e^{4\pi b
\widetilde{u}^2}-1)^2+(e^{4\pi a v_L^2}-1)^2+(e^{4\pi b
\widetilde{u}^2}-1)^2\right]
\end{eqnarray*}
and,

$$
\| (e^{4\pi (v_L+\widetilde{u})^2(t)}-1)\|^2_{L^2(\Omega)}\leq
3(I_1(t)+I_2(t)+I_3(t)),
$$
where we set\\
$$I_1(t)=\displaystyle\int_{\Omega}(e^{4\pi b
\widetilde{u}^2(t,x)}-1)^2\;dx\quad,\quad
I_2(t)=\displaystyle\int_{\Omega}(e^{4\pi a v_L^2(t,x)}-1)^2\;dx$$

$$\mbox{and\quad}I_3(t)=\displaystyle\int_{\Omega}(e^{4\pi a
v_L^2(t,x)}-1)^2(e^{4\pi b \widetilde{u}^2(t,x)}-1)^2\;dx.$$
\\
Thanks to the continuity in time of $v_L$ and $\widetilde{u}$ and
the fact that $\|\nabla u_0\|_{L^2(\Omega)}< 1$, one can choose
$\varepsilon_1$ arbitrary small (to be fixed later) and take $0<A^2:=\frac 1 2(1+\|u_0\|^2_{H^1_0(\Omega)})<1$ to find a time $T>0$
such that for all $t\in[0,T]$

$$
\|\widetilde{u}(t,.)\|_{H^1_0(\Omega)}\leq
\varepsilon_1\quad\mbox{and}\quad\|v_L(t,.)\|_{H^1_0(\Omega)}\leq
A.
$$
Combining H\"older's inequality, the log estimate \eqref{il} and the
fact that for all $t\in[0,T]$,
$$
{\| v_L(t,.)\|}_{H_0^1({\Omega})}\leq A
$$
with the monotonicity of the function $x\longmapsto
x^2\log(B_1+\displaystyle\frac{B_2}{x})$ lead to

\begin{eqnarray*}
I_2(t)&\leq& e^{4\pi a{\Vert
v_L(t,.)\Vert}^2_{L^{\infty}(\Omega)}}\int_{\Omega}(e^{4\pi a
v_L^2(t,x)}-1)\;dx\\
&\leq&\left(C_\lambda+\frac{{\Vert
v_L(t,.)\Vert}_{C^{1/8}}}{A}\right)^\beta\int_{\Omega}(e^{4\pi a
v_L^2(t,x)}-1)\;dx,
\end{eqnarray*}
where we set $\beta={4\pi a A^2\lambda}.$\\
Now we choose $\varepsilon>0$ such that $4\pi a A^2<4\pi$. Then, by Moser-Trudinger inequality \eqref{MT1}

\begin{eqnarray*}
\int_{\Omega}(e^{4\pi a v_L^2(t,x)}-1)\;dx\leq C(\Omega,A)
\end{eqnarray*}\\
and therefore

$$
I_2(t)\leq C(\Omega,A)\left(C_\lambda+\frac{{\Vert
v_L(t,.)\Vert}_{C^{1/8}}}{A}\right)^\beta.
$$
Now, using \eqref{ob3} we have
$$
I_1(t)\leq \displaystyle\int_{\Omega}e^{8\pi b
\widetilde{u}^2(t,x)}-1\;dx.
$$
Choosing $\varepsilon_1>0$ such that $2 b {\varepsilon_1}^2\leq1$, then again by Moser-Trudinger inequality \eqref{MT1}, we have
$$
I_1(t)\leq C(\Omega).
$$
Now applying H\"{o}lder inequality and \eqref{ob3}, we obtain

\begin{eqnarray*}
I_3(t)&\leq&\left(\int_{\Omega}(e^{4\pi a
v_L^2(t,x)}-1)^{2a}\;dx\right)^{\frac{1}{a}}\left(\int_{\Omega}(e^{4\pi
b \widetilde{u}^2(t,x)}-1)^{2b}\;dx\right)^{\frac{1}{b}}\\
&\leq&e^{4\pi a{\Vert
v_L(t,.)\Vert}^2_{L^{\infty}(\Omega)}}\left(\int_{\Omega}(e^{4\pi
{a^2}v_L^2(t,x)}-1)\;dx\right)^{\frac{1}{a}}\left(\int_{\Omega}(e^{4\pi
2b^{2}\widetilde{u}^2(t,x)}-1)\;dx\right)^{\frac{1}{b}},
\end{eqnarray*}
and similarly as before, we estimate

$$
I_3(t)\leq C \left(C_\lambda+\frac{{\Vert
v_L(t,.)\Vert}_{C^{1/8}}}{A}\right)^\beta.
$$
Consequently,

$$
{\Vert e^{4\pi (v_L+\widetilde{u})^2}-1\Vert}_{L^2(\Omega)}\leq
C\left[1+\left(C_\lambda+\frac{{\Vert
v_L(t,.)\Vert}_{C^{1/8}}}{A}\right)^{\beta'}\right]
$$
where ${\beta'}=2\pi a\lambda A^2$. Therefore,

\begin{eqnarray*}
\|F(u)\|_{L^1_T(L^2(\Omega))}&\leq& C \int_0^T
\left(1+\left(C_\lambda+\frac{{\Vert
v_L(t,.)\Vert}_{C^{1/8}}}{A}\right)^{\beta'}\right)\;dt\\
&=& C\left[T+\int_0^T\left(C_\lambda+\frac{{\Vert
v_L(t,.)\Vert}_{C^{1/8}}}{A}\right)^{\beta'}\;dt\right].
\end{eqnarray*}
Choosing $\lambda$ such that $\beta'<8$ and applying H\"older
inequality with $p=8/\beta'$, we obtain

\begin{eqnarray*}
\|F(u)\|_{L^1_T(L^2(\Omega))}&\lec&
\left[T+T^{1-\frac{\beta'}{8}}\left\|C_\lambda+\frac{{\Vert
v_L(t,.)\Vert}_{C^{1/8}}}{A}\right\|_{L^8_T}^{\beta'}\right]\\
&\lec&
\left[T+T^{1-\frac{\beta'}{8}}\left(T^{\frac{\beta'}{8}}+\left(\frac{\|v_L(t)\|_{{L^8_T}(C^{1/8}(\Omega))}}{A}\right)^{\beta'}\right)\right]\\
&\lec&\left[T+T^{1-\frac{\beta'}{8}}\left(\frac{\|v_L(t)\|_{{L^8_T}(C^{1/8}(\Omega))}}{A}\right)^{\beta'}\right]:=C(T).
\end{eqnarray*}
\endproof
Now we prove the local existence result.\\
\textbf{{Proof of Theorem \ref{loc}}}. The proof is divided into two steps\\\\
\textbf{Step 1:} The existence in $\mathcal E_T.$

We write the solution $u$ of problem \eqref{C1} as
$$u=v+v_L$$
with as before $v_L$ solves the free wave equation
with the same initial data $(u_0,u_1)$ and $v$ solves the following perturbed problem

\begin{equation}\label{C2}
\left\{
\begin{array}{llll}

\Box{v}=-(v+v_L)(e^{4\pi(v+v_L)^2}-1)\\
{v}(0,x)=0\\
{\p}_{t}{v}(0,x)=0\\
v|_{[0,T]\times\p \Omega}=0.
\end{array}
\right.
\end{equation}
Define the map $\phi:\mathcal
E_T\longrightarrow \mathcal E_T;\;\;v \longmapsto \tilde{v}, $ where
$\widetilde{v}$ satisfies
\begin{equation}\label{C3} \left\{
\begin{array}{llll}
\Box\widetilde{v}=-(v+v_L)(e^{4\pi(v+v_L)^2}-1)\\
\widetilde{v}(0,x)=0\\
{\p}_{t}\widetilde{v}(0,x)=0\\
\widetilde{v}|_{[0,T]\times\p \Omega}=0.
\end{array}
\right.
\end{equation}
We claim that for $T$ small enough, the map $\phi$ is well defined from $\mathcal E_T$ into itself and is a
contraction. Indeed, consider $v_1$ and $v_2$ in $\mathcal E_T$ and set
$$
u_1=v_1+v_L\quad u_2=v_2+v_L.
$$
Using the energy and Strichartz estimates we have

$$
{\Vert \phi(v_1)-\phi(v_2)\Vert}_T\leq C {\Vert
f(u_1)-f(u_2)\Vert}_{L^1_T(L^2(\Omega))}.
$$
Since $u_1$ and $u_2$ are two elements of $\mathcal E_T$ satisfying
$u_1(0,x)=u_2(0,x)=u_0(x)$ and $\|u_0\|_{H^1_0}<1$, then there exist $0<A<1$ and a positive real
number $T_0$ such that for any $0\leq t\leq T_0$,
$$
{\Vert u_1(t)\Vert}_{H_0^1(\Omega)}\leq A\quad \mbox{and}
\quad{\Vert u_2(t)\Vert}_{H_0^1(\Omega)}\leq A.
$$
Thanks to Lemma \ref{lem1}, there exist $0<\gamma<8$ such that for
any $T\in[0,T_0]$

\begin{eqnarray*}
{\Vert f(u_1)-f(u_2)\Vert}_{L^1_T(L^2(\Omega))}&\leq& C{\Vert
u_1-u_2\Vert}_{T} \left( T+T^{\frac{8-\gamma}{8}}
\left(\left(\frac{{\Vert u_1\Vert}_T}{A}\right)^{\gamma}+
\left(\frac{{\Vert u_2\Vert}_T}{A}\right)^{\gamma}\right)\right)\\
&\leq& C{\Vert v_1-v_2\Vert}_{T} \left[T+T^{\frac{8-\gamma}{8}}
 \left(\left(\frac{{\Vert
 u_1\Vert}_T}{A}\right)^{\gamma}+
 \left(\frac{{\Vert
 u_2\Vert}_T}{A}\right)^{\gamma}\right)
 \right].
 \end{eqnarray*}
Let $C(T)=C\left[T+T^{\frac{8-\gamma}{8}}
 \left(\left(\frac{{\Vert
 u_1\Vert}_T}{A}\right)^{\gamma}+
 \left(\frac{{\Vert
 u_2\Vert}_T}{A}\right)^{\gamma}\right)
 \right]$, we have

$$
{\Vert \phi(v_1)-\phi(v_2)\Vert}_T \leq  C(T){\Vert
v_1-v_2\Vert}_{T}.
$$
\\
So, for $T$ small enough, we have $C(T)\leq 1/2$ implying that
$\phi$ is a contraction map. Taking $v_2=0$ shows that $\phi$ is well defined.\\

\textbf{Step 2:} Uniqueness in the energy space.\\
Let $U_1$ and $U_2$ be in $C([0 , T] , H_0^1(\Omega))\cap C^1([0 ,
T],L^2(\Omega))$ two solutions of problem \eqref{C1} having the
same initial data $(u_0, u_1 )$. Let $w=U_1-U_2$, then $w$ satisfies
\begin{equation*}
\left\{
\begin{array}{llll}

\Box{w}=U_2(e^{4\pi U_2^2}-1)-U_1(e^{4\pi U_1^2}-1)\\
{w}(0,x)=0\\
{\p}_{t}{w}(0,x)=0\\
w|_{[0,T]\times\p\Omega}=0.
\end{array}
\right.
\end{equation*}\\
\\
In the sequel we shall prove the existence of a continuous function $C(\cdot)$ defined on $[0,T]$, vanishing at $t=0$ and such that
$$
\|w\|_E\leq C(T)\|w\|_E,
$$
where ${\| w\|}_E=\displaystyle\sup_{t\in[0,T]}
 (\Vert w(t,.)\Vert_{H_0^1(\Omega)}+\Vert{\p_{t}w(t,.)}\Vert_{L^2(\Omega)})$.\\
Using the energy estimate, the mean value Theorem and \eqref{r}, we
have
  \begin{eqnarray*}
  {\| w\|}_E
  &\leq& C{\|
 U_2(e^{4\pi U_2^2}-1)-U_1(e^{4\pi
 U_1^2}-1)\|}_{L^1_T(L^2(\Omega))}\\
 &\leq&C{\Vert
 U_2(e^{4\pi U_2^2}-1)-(U_2+w)(e^{4\pi
 (U_2+w)^2}-1)\Vert}_{L^1_T(L^2(\Omega))}\\
 &\leq&C{\Vert
 w(e^{4\pi(1+\varepsilon) \overline{w}^2}-1)
\Vert}_{L^1_T(L^2(\Omega))},
  \end{eqnarray*}
where, $\overline{w}=(1-\theta)(w+U_2)+\theta U_2=(1-\theta)w+U_2.$ Thanks to H\"older inequality, the Sobolev embeddings and
\eqref{ob3}, we have
\begin{eqnarray*}
{\Vert
w(t)(e^{4\pi(1+\varepsilon)\overline{w}^2(t)}-1)\Vert}^{2}_{L^2(\Omega)}&\leq&{\Vert
w(t)\Vert}^2_{L^{2+2/\varepsilon}(\Omega)}{\Vert
(e^{4\pi(1+\varepsilon)\overline{w}^2(t)}-1)^2\Vert}_{L^{1+\varepsilon}(\Omega)}\\
\\
&\leq& {\Vert w(t)\Vert}^2_{H_0^1(\Omega)}{\Vert
(e^{4\pi(1+\varepsilon)^2\overline{w}^2(t)}-1)\Vert}^{2/{(1+\epsilon)}}_{L^2(\Omega)}.
\end{eqnarray*}
By continuity in time of $w$ and $U_2$ and the fact that
$w(0,x)={\p}_{t}w(0,x)=0$ and $U_2(0,x)=u_0(x)$ with $\|\nabla
u_0\|_{L^2(\Omega)}<1$, there exist a positive real number $T_1$ such
that, for any $t\in [0,T_1]$

$$
\|w(t)\|_{H_0^1(\Omega)}\leq \varepsilon\quad \mbox{and}\quad
\|U_2(t)\|_{H^1_0(\Omega)}\leq A.
$$
Arguing as in the proof of Lemma \ref{lem2}, we obtain
\begin{eqnarray*}
{\Vert
e^{4\pi(1+\varepsilon)^2\overline{w}^2(t,\cdot)}-1\Vert}_{L^2}
&\lec&
\big(1+\|e^{4\pi(1+\varepsilon)^2 a U_2^2(t,\cdot)}-1\|_{L^2(\Omega)}\\
&+&\|e^{4\pi(1+\varepsilon)^2 a^2
U_2^2(t,\cdot)}-1\|_{L^2(\Omega)}\big).
\end{eqnarray*}
Finally, for any $0<T\leq T_1$

\begin{eqnarray*}
\displaystyle\int_0^T{\Vert
(e^{4\pi(1+\varepsilon)^2\overline{w}^2(t)}-1)\Vert}_{L^2}^{\frac{1}{1+\varepsilon}}\;dt
&\lec&T+\int_0^T\|(e^{4\pi(1+\varepsilon)^2 a
U_2^2}-1)\|^{\frac{1}{1+\varepsilon}}_{{L^2(\Omega)}}\;dt\\ &+&\int_0^T\|(e^{4\pi(1+\varepsilon)^2
a^2 U_2^2}-1)\|^{\frac{1}{1+\varepsilon}}_{L^2(\Omega)}\;dt .
\end{eqnarray*}
To estimate the last two terms in the above right-hand side,
we use Lemma \ref{lem2}. Hence

$$
\int_0^T{\Vert
(e^{4\pi(1+\varepsilon)^2\overline{w}^2(t)}-1)\Vert}_{L^2}^{\frac{1}{1+\varepsilon}}\;dt\leq
C(T)
$$
\\
and finally we have
\begin{eqnarray*}
{\Vert
 w(e^{4\pi(1+\varepsilon) \overline{w}^2}-1)
\Vert}_{L^1_T(L^2(\Omega))}&\leq&
\sup_{t\in[0,T]}\|w(t)\|_{H_0^1(\Omega)}\int_0^T{\Vert
(e^{4\pi(1+\varepsilon)^2\overline{w}^2(t)}-1)\Vert}_{L^2}^{\frac{1}{1+\varepsilon}}\;dt\\
\\
&\leq& C(T)\|w\|_E
\end{eqnarray*}
as desired.
\endproof
\section{The global existence }
Theorem \ref{loc} guarantees that in the subcritical and critical
cases, there exists a unique local solution to the Cauchy problem
\eqref{C1}. In this section we propose to extend the local existence
result to global one (in time). We start by the subcritical case and prove Theorem \ref{csc}.

\subsection{The subcritical case: Proof of Theorem \ref{csc}}

\begin{proof}
We have $E_0<1$, so in particular $\|\nabla u_0\|_{L^2(\Omega)}<1$. Then,
according to the local theory (Theorem \ref{loc}), there exist a
unique maximal solution $u$ in the space $\mathcal E_T{^\star}$
where $0<T^\star\leq +\infty$ is the lifespan of $u$. The fact that
$T^*=+\infty$ is then an immediate consequence of the energy
conservation
$$
\sup_{0<t<T^\star}\|\nabla u(t,\cdot)\|_{L^2(\Omega)}\leq E(u,t)=E_0<1,
$$
and the fact that $T^*$ depends upon $1-\|\nabla u_0\|_{L^2}^2$.
\end{proof}

The proof in the critical case is more subtle. Indeed, we need to
show that concentration cannot occur close to $T^*$. We combine
ideas from \cite{sli moh nad1} and \cite{BLP}. However, it is
important to point out here that our proof is simpler than that one
of Burq-Lebeau-Planchon in \cite{BLP} for the quintic energy critical
equation in dimension three. This is because for our purpose, we
only use the multipliers $u$ and $\partial_t u$. The multiplier
$x\cdot\nabla u$ requires more careful study since it generates other boundary terms
but it is not needed here. See \cite{BLP} for complete details.

\subsection{The critical case : Proof of Theorem \ref{cc}}

\begin{proof}
Let $u$ be the unique maximal solution to the Cauchy problem \eqref{C1} in the space $\mathcal E_T{^\star}$. We show that if $T^\star$ is finite then we have a contradiction. We start by showing some properties of the maximal solution $u$ in the critical case.
\begin{prop}\label{PC12*}
The maximal solution $u$ verifies

\begin{eqnarray}\label{PC11}
\displaystyle\limsup_{t\rightarrow T^\star}\|\nabla
u(t)\|_{L^2(\Omega)}=1,
\end{eqnarray}
and

\begin{eqnarray}\label{PC12}
u(t)\displaystyle\stackrel{t\rightarrow T^\star}{\longrightarrow} 0
\quad\mbox{in}\; L^2(\Omega).
\end{eqnarray}
\end{prop}
\emph{Proof.} Using \eqref{enr}, we have for all $0\leq t<T^\star$,
$$
\|\nabla u(t)\|^2_{L^2(\Omega)}+\|\p_t
u(t)\|^2_{{L^2}(\Omega)}+\int_{\Omega}\frac{ e^{{4\pi
u^2(t,x)}}-1-4\pi u^2(t,x)}{4\pi}\;dx=1.
$$
Hence ,

$$
\displaystyle\limsup_{t\rightarrow T^\star}\|\nabla
u(t)\|_{L^2(\Omega)}\leq 1.
$$
Assuming that $\displaystyle\limsup_{t\rightarrow T^\star}\|\nabla
u(t)\|_{L^2(\Omega)}=\ell<1$, then for
$\varepsilon:=\frac{1-\ell}{2}$, one can find a time $t_0$ such that for
all $0<t_0<t<T^\star$, we have
$$
\|\nabla u(t)\|_{L^2(\Omega)}\leq \frac{\ell+1}{2}.
$$
Moreover, by continuity, there exists a time $t_1$ in the interval
$[0,t_0]$, such that

$$
\displaystyle\sup_{0\leq t \leq t_0}\|\nabla u(t)\|_{L^2(\Omega)}=\|\nabla
u(t_1)\|_{L^2(\Omega)}<1.
$$
Hence

$$
\displaystyle\sup_{0\leq t<T^\star}\|\nabla u(t)\|_{L^2(\Omega)}<1.
$$
Consequently, $u$ can be extended beyond the time $T^\star$, a contradiction. \\
Now, let us show \eqref{PC12}. We consider a sequence $(t_n)$
converging to $T^\star$ as $n\longrightarrow+\infty$. We start by
proving that $u_n:=u(t_n)$ is a Cauchy sequence in $L^2(\Omega)$.
Indeed,

\begin{eqnarray*}
\|u(t_n)-u(t_m)\|_{L^2(\Omega)}&\leq&|t_n-t_m|\displaystyle\sup_{\tau\in[0,T^\star)}\|\p_tu(\tau)\|_{L^2(\Omega)}\\
&<&|t_n-t_m|,
\end{eqnarray*}
which can be made arbitrary small. Thus, there exists $\overline{u}$ in $L^2(\Omega)$ such that $u(t)$ converges to $\overline{u}$ in $L^2(\Omega)$ as $t\longrightarrow T^\star$. Now, we prove that $\overline{u}=0$. Using \eqref{enr} and Fatou's Lemma, we have

\begin{eqnarray*}
\displaystyle\limsup_{t\rightarrow T^\star}\|\nabla
u(t)\|^2_{L^2(\Omega)}-1 &\leq & -\displaystyle\liminf_{t\rightarrow
T^\star}\|\p_t
u(t)\|^2_{L^2(\Omega)}-\displaystyle\liminf_{t\rightarrow
T^\star}\int_{\Omega}\frac{e^{4\pi u^2(t,x)}-1-4\pi u^2(t,x)}{4\pi}\;dx\\
& \leq& -\displaystyle\liminf_{t\rightarrow T^\star}\|\p_t
u(t)\|^2_{L^2(\Omega)}-\int_{\Omega}\displaystyle\liminf_{t\rightarrow
T^\star}\frac{e^{4\pi u^2(t,x)}-1-4\pi u^2(t,x)}{4\pi}\;dx.
\end{eqnarray*}
\\\\
By \eqref{PC11}

$$
\displaystyle\liminf_{t\rightarrow T^\star}\|\p_t
u(t)\|^2_{L^2(\Omega)}+\int_{\Omega}\displaystyle\liminf_{t\rightarrow
T^\star}\frac{e^{4\pi u^2(t,x)}-1-4\pi u^2(t,x)}{4\pi}\;dx\leq 0,
$$
which implies

$$
\liminf_{t\rightarrow T^\star}(e^{4\pi u^2(t,x)}-1-4\pi u^2(t,x))=0
$$
Therefore, $\overline{u}=0.$ This completes the proof of Proposition \ref{PC12*}.\\
\endproof
Now we construct a sort of ``critical element" in the sense that all its energy concentrates in the backward light cone issued from a point. Since the equation is invariant under time translation, in the sequel we will
assume that $T^\star=0$.
\begin{prop}
\label{proposition}
Let $u$ be the maximal solution of problem \eqref{C1}. Then, there
exists a point $x^\star$ in $\overline{\Omega}$ such that, for all $t<0$
\begin{equation}\label{cc1}
supp\;\nabla u(t,\cdot)\subset B(x^\star, -t)\cap\bar\Omega, \quad supp\;\partial_t u(t,\cdot)\subset B(x^\star, -t)\cap\bar\Omega.
\end{equation}
\end{prop}

The proof goes along the same lines as in \cite{sli moh nad1}. For the convenience of the reader, we sketch it here.
\begin{proof}
\textbf{Claim 1:}  There exists a point  $x^*$ in $\bar\Omega$ such that
for all  $r > 0$, we have
\begin{equation}
\label{ec} \limsup_{t\longrightarrow 0^- }\;\;\int_{\{x;\;|x-x^*|\leq r\}\cap\bar\Omega}\;\;|\nabla u(t)|^2\;\;dx=1.
\end{equation}
Indeed, by contradiction and as in \cite{sli moh nad1}, there exist
two positive real numbers $r$ and $\eta$ such that for any $x\in \bar\Omega$ we have

\begin{eqnarray}
\label{Lim}
\limsup_{t\longrightarrow 0^- }\;\;\int_{\{y;\;|x-y|\leq r\}\cap\bar\Omega}\;\;e(u)(t,y)\;\;dy\leq
1-\eta.
\end{eqnarray}
Now let $x\in\bar\Omega$ and define the cut-off function
$\varphi_x$ by $0\leq\varphi_x\leq1$, $\varphi_x\equiv 1$ in $B(x,r/2)\cap\bar\Omega$ and $\varphi_x\equiv
0$ outside  $B(x,r)\cap\bar\Omega$.
Obviously, from (\ref{Lim}) and Proposition \ref{PC12*}, we have
$$\limsup_{t\longrightarrow 0^- }\;\;\int_{\{y;\;|x-y|\leq r\}\cap\bar\Omega}\;\;  e(\varphi_x u,\varphi_x \partial_t u )(t)
\;\;dy\leq
1-\eta.$$
Now choose a time $t_1>T^*-r/8$ such that
$$
\int_{\{y;\;|x-y|\leq r\}\cap\bar\Omega}
\;\;  e(\varphi_x u,\varphi_x \partial_t u )(t_1)
\;\;dy\leq
1-\eta/2.
$$
From the local theory (Theorem \ref{loc}), one can solve globally in time problem \eqref{C1} with the initial data
$(\varphi_xu(t_1,\cdot),\varphi_x\partial_t u(t_1,\cdot)) $. By the finite speed of propagation, we deduce that
$u$ can be continued in the backward light cone of vertex
$(x,t_1+r/2)$. Since the set $\bar\Omega$ is compact, then we can extract a finite covering from $\bar\Omega=\cup_{x\in\bar\Omega}B(x,r)\cap\bar\Omega$. This implies that $u$ can be continued beyond its lifetime
$T^*$ which is a contradiction.

\textbf{Claim 2:} We have the following
\begin{equation}
\label{ec5} \lim_{t\longrightarrow 0^- }\;\;\dint_{\{x,\;|x-x^*|\leq -t\}\cap\bar\Omega}\;\;|\nabla u(t)|^2\;\;dx=1.
\end{equation}
\begin{equation}
\label{ec6} \forall\;\; t<0,\qquad\int_{\{x,\;|x-x^*|\leq -t\}\cap\bar\Omega}
\;\;e(u(t))\;dx=1.
\end{equation}
Indeed, without loss of generality, we can assume that $x^*=0$. The proof of (\ref{ec5}) is
straightforward. Suppose that (\ref{ec5}) is  false. Then, there
exists a sequence of negative real number $(t_n)$ tending to zero such that
$$
\forall\,\,n\in\N,\quad
\dint_{|x|\leq -t_n}\;\;|\nabla u (t_n)|^2\;\;dx\leq 1-\eta\quad
\hbox{for some} \     0<\eta<1.
$$
Then, arguing as in the proof
of the previous claim,  the solution can be continued beyond $T^*$,  a contradiction.  To prove
(\ref{ec6}), fix $\varepsilon>0$. By (\ref{ec5}), there
exists a time $t_{\varepsilon}<0$ such that $\dint_{|x|\leq
-t}\;\;|\nabla
u(t)|^2\;\;dx\geq 1-\varepsilon$ for $t_{\varepsilon}\leq t<0$.
By the finite speed of propagation,
 we deduce that $$\forall\;\;
t<0,\;\;\;\int_{|x|\leq -t}\;\;e(u)(t)\;dx\geq
1-\varepsilon.$$ Letting $\varepsilon$ go to zero, we obtain the
desired result.

Now, the proof of Proposition \ref{proposition} is immediate. If for a fixed $t<0$, the support property is not satisfied, then there exist $\varepsilon_0>0$ and $\eta_0>0$ such that for all $x_0\in\bar\Omega$, we have

$$
\int_{   \{x,\;|x-x_0|\geq(1+\eta_0)(-t)\}\cap\bar\Omega   }  |\nabla u(x,t)|^2+|\partial_t u(x,t)|^2\;dx\geq\varepsilon_0.
$$
But for $x_0=x^*$, the above inequality together with \eqref{ec6}  contradict the fact that the $E(u,t)=1$.
\end{proof}



\emph{\textbf{Proof of Theorem \ref{cc}.}}\\ Multiplying equation \eqref{C1} by $ 2\p_t u$, we obtain
\begin{equation}\label{eq 1}
\p_t(e(u))-\mbox{div}_x( 2\p_t u \cdot\nabla u)=0,
\end{equation}
where the energy density $e(u)$ is defined by \eqref{e}.\\
Integrating \eqref{eq 1} over the backward truncated cone $K^T_S$
($S<T<0$), we get
\begin{eqnarray}\label{eq12}
\int_{K^T_S}\mbox{div}_{t,x}{\vec B}(t,x)\;dx\;dt=0,
\end{eqnarray}
where
$$
\vec B=(B_0,B_1,B_2),\quad \; B_0=e(u)\;\;\mbox{and}\;
B_j= -2\p_t u\frac{\p u}{\p x_j},\;j=1,2.
$$
Thanks to Stokes formula, we obtain

\begin{eqnarray*}
\int_{D(T)}e(u)(T)\;dx-\int_{D(S)}e(u)(S)\;dx&-&\int_{\{([S,T]\times\p\Omega)\}\cap K_S^T}\nu(x)\cdot(2\p_t u\nabla u)d\sigma\\
&+&\displaystyle\frac{1}{\sqrt{2}}\int_{M_S^T}
\{|\p_tu\frac{x}{|x|}+\nabla u|^2+\frac{e^{4\pi
u^2}-1-4\pi u^2}{4\pi}\}\;d\sigma=0,
\end{eqnarray*}
here $M_S^T$ defined by \eqref{mst} and $\nu(x)$ is the exterior normal vector to $\Omega$ at point
$x$. Taking into account the Dirichlet boundary condition, we have

\begin{eqnarray}\label{es1}
\int_{D(S)}e(u(S))\;dx-\int_{D(T)}e(u(T))\;dx=\int_{M_S^T}\{|\p_tu
\frac{x}{|x|}+\nabla u|^2+\frac{e^{4\pi u^2}-1-4\pi u^2}{4\pi}\}\frac{d\sigma}{\sqrt{2}}.
\end{eqnarray}
Now, multiplying equation \eqref{C1} by $2u$, integrating
over the backward truncated cone $K_S^T$ and using Stokes formula given the Dirichlet condition, we obtain

\begin{eqnarray}\label{B}
\int_{D(T)}\p_tu(T)u(T)\;dx-\int_{D(S)}\p_tu(S)u(S)\;dx
\end{eqnarray}

\begin{eqnarray}
+\displaystyle\frac{1}{\sqrt{2}}\int_{M_S^T}
(\p_tu+\nabla u\cdot\frac{x}{|x|})u\;d\sigma
+\int_{K^T_S}(|\nabla u|^2-|\p_tu|^2+u^2(e^{4\pi
u^2}-1))\;dx\;dt=0.\nonumber
\end{eqnarray}
Thanks to \eqref{cc1}, identity \eqref{es1} implies that
\begin{eqnarray}\label{flux}
\frac{1}{\sqrt{2}}\int_{M_S^T}\{|\p_tu \frac{x}{|x|}+\nabla
u|^2+\frac{e^{4\pi u^2}-1-4\pi u^2}{4\pi}\}\;d\sigma=0.
\end{eqnarray}
Since $u(t)\longrightarrow 0\;$ in $L^2(\Omega)$ and $\|\nabla
u(t)\|_{L^2(\Omega)}\longrightarrow1$ as $t $ goes to 0, the energy
identity \eqref{enr} implies that
\begin{eqnarray}\label{der}
\p_t u(t)\longrightarrow0 \mbox{ in }L^2(\Omega).
\end{eqnarray}
 Letting $T$ go to zero in \eqref{B},
using \eqref{der} and \eqref{flux}, we have
\begin{eqnarray*}
-\int_{D(S)}\p_tu(S)u(S)\;dx +\int_{K^0_S}(|\nabla
u|^2-|\p_tu|^2+u^2(e^{4\pi u^2}-1))\;dx\;dt=0.
\end{eqnarray*}
Multiplying the above identities by $\displaystyle\frac{-1}{S}$,
we deduce that

$$
\int_{D(S)}\p_tu(S)\frac{u(S)}{S}\;dx\leq
\frac{1}{S}\int_{K^0_S}|\nabla
u|^2\;dx\;dt-\frac{1}{S}\int_{K^0_S}|\p_tu|^2\;dx\;dt.
$$
Thanks to the mean value Theorem, there exist $t_0\in]S,0[$ such
that

$$
\frac{1}{S}\int_{K^0_S}|\nabla u|^2\;dx\;dt=-\int_{|x-x^\star|\leq -t_0}|\nabla
u(t_0,x)|^2\;dx.
$$
So, using \eqref{PC11}

$$
\frac{1}{S}\int_{K^0_S}|\nabla u|^2\;dx\;dt\stackrel{S\rightarrow
0^-}{\longrightarrow -1}.
$$
Similarly

$$
\frac{1}{S}\int_{K^0_S}|\p_t u|^2\;dx\;dt\stackrel{S\rightarrow
0^-}{\longrightarrow 0}.
$$
Moreover, since $
\left|\frac{u(S)}{S}\right|=\left|\frac{1}{S}\int_0^S\p_tu(\tau)\;d\tau\right|$, then $(\frac{u(S)}{S})$ is bounded in $L^2(\Omega).$ H\"older inequality combined to the above result imply
$$
\int_{D(S)}\p_tu(S)\frac{u(S)}{S}\;dx\stackrel{S\rightarrow
0^-}{\longrightarrow 0},
$$
leading to $0\leq -1$, a contradiction.
\end{proof}
 \section{Ill-posedness in the supercritical case}
In this section we prove the instability result given by Theorem \ref{ins}. The construction is similar to that one in
Proposition \ref{supercriticalST}. However here, we have to consider the nonlinear problem and not just the linear one. In particular, we will show that the solution to the ODE (the nonlinear wave equation without the diffusion term) is a ``perturbation" of the cosine function. We construct a slightly supercritical initial data given through the same functions $f_k$ as in \eqref{lesfk}. The concentration presented in the data yields fast periodic oscillations in the ODE regime. Moreover, the special form of the data and the finite speed of propagation allow us to conclude that solutions of the P.D.E. and the ODE coincide in a backward light cone.\\

$\bullet$ \textbf{\emph{Step 1: Construction of the initial
data.}}\\ Without loss of generality, we can assume that
$0\in\Omega$. Choose $0<\eta<1$ small enough such that the ball
$B(0,\eta)\subset \Omega$. For $k\geq 1$, let $v_k$ solve

$$
\Box v_k+v_k(e^{4\pi v_k^2}-1)=0,\quad v_k(0,x)=(1+\frac{1}{k})f_k(\frac
x \eta),\quad {\p}_{t}v_k(0,x)=0,\quad v_k|_{\p\Omega}=0
$$
and $w_k$ the solution of
$$
\Box w_k+w_k(e^{4\pi w_k^2}-1)=0,\quad w_k(0,x)=f_k(\frac x \eta),\quad
{\p}_{t}w_k(0,x)=0,\quad w_k|_{\p\Omega}=0.
$$
Since,
\begin{equation}
\|\nabla f_k(\frac \cdot \eta)\|_{L^2(\Omega)}^2=\int_{\eta
e^{-k/2}\leq|x|\leq \eta}\frac{1}{k\pi|x|^2}\;dx
=\frac{2}{k}\int_{\eta e^{-k/2}}^\eta \frac{dr}{r}=1,
\end{equation}
we easily verify that given $\varepsilon>0$, then using Poincar\'e inequality
\begin{eqnarray*}
\|v_k(0)-w_k(0)\|^2_{H_0^1(\Omega)}+\|\p_t v_k(0)-\p_t
w_k(0)\|^2_{L^2(\Omega)}=\frac{1}{k^2}\|f_k(\frac \cdot
\eta)\|^2_{H_0^1(\Omega)}\leq \frac{C}{k^2}\leq \varepsilon,
\end{eqnarray*}
for $k$ large enough. Therefore $w_k$ and $v_k$ satisfy \eqref{ins1}. Now, we will show
that the initial data associated to $v_k$ and $w_k$ are slightly
supercritical.
$$
E(w_k,0)=\|\nabla
f_k(\frac\cdot\eta)\|^2_{L^2(\Omega)}+\frac{1}{4\pi}\int_{\Omega}e^{4\pi
f_k^2(\frac \cdot \eta)}-1-4\pi f_k^2(\frac \cdot \eta)\;dx\leq1+\frac{1}{4\pi}\int_{\Omega}(e^{4\pi
f_k^2(\frac \cdot \eta)}-1)\;dx.
$$
But,
\begin{eqnarray*}
\frac{1}{4\pi}\int_{\Omega}(e^{4\pi f_k^2(\frac \cdot \eta)}-1)\;dx
&=&\frac{1}{4\pi}\left(\int_{\eta e^{-k/2}\leq |x|\leq
\eta}(e^{\frac{4}{k}\log^2(\frac{|x|}{\eta})}-1)\;dx
+\int_{|x|\leq \eta e^{-k/2}}(e^k-1)\;dx\right)\\
&=&\frac{\eta^2}{2}\int_{e^{-k/2}}^1 r(e^{\frac{4}{k}\log^2 r}-1)\;dr+\frac{e^k-1}{2}\int_0^{\eta e^{-k/2}}r\;dr.\\
&=& \frac{\eta^2}{2}\int_{e^{-k/2}}^1 r e^{\frac{4}{k}\log^2 r}\;dr,
\end{eqnarray*}
and to estimate the last integral, we use the following Lemma (see
\cite{sli moh nad2}).

\begin{lem}\label{l1}
For any $a\geq 1$ and $k\in \N$,
\begin{eqnarray*}
I(a,k):=\int_{e^{-k/2}}^1 r e^{\frac{4a^2}{k}\log^2r}\;dr\leq
2e^{(a^2-1)k}.
\end{eqnarray*}
\end{lem}
Applying the above Lemma with $a=1$, we get
$$
\frac{\eta^2}{2}\int_{e^{-k/2}}^1 r e^{\frac{4}{k}\log^2 r}\;dr\leq
\eta^2.
$$
Hence,  for $k$ large enough, $E(w_k,0)\leq 1+\eta^2.$ \\
Similarly, we prove that $0<E(v_k,0)-1\leq
C\eta^2e^{2+\frac{1}{k}}$. Therefore, for $k$ large enough
$$
0<E(v_k,0)-1\leq 3\eta^2e^{3}.
$$

$\bullet$\textbf{ \emph{Step 2: Approximation}.}\\
Let $\phi_k$ and $\psi_k$ be the two solutions of the following
ordinary differential equation (O.D.E.)

\begin{eqnarray}
\ddot{y}+y(e^{4\pi y^2}-1)=0,
\end{eqnarray}
with initial data

$$
\phi_k(0)=(1+\frac{1}{k})\sqrt{\frac{k}{4\pi}}\quad,\quad
\dot{\phi_k}(0)=0,
$$
and

$$
\psi_k(0)=\sqrt{\frac{k}{4\pi}}\quad,\quad \dot{\psi_k}(0)=0.
$$
Since $v_k=\phi_k$ and $w_k=\psi_k$ on the ball
$B=\{(x,t=0)\;:\;|x|\leq \eta e^{-k/2}\}$ in the hyperplane $t=0$,
then by finite speed of propagation $v_k=\phi_k$ and $w_k=\psi_k$ in
the backward light cone
$$
K=\{(x,t)\;/\;t=\alpha \eta e^{-k/2}\quad |x|\leq (1-\alpha)\eta
e^{-k/2}\;;0\leq\alpha\leq 1\}.
$$
$\bullet$\textbf{\emph{Step 3: Decoherence.}}\\\\
We start by recalling the following result (for example, see section III.5 from \cite{AL}).
\begin{lem}
Let $F :\R\longrightarrow \R$ be a regular function and consider the
following O.D.E.
\begin{equation}\label{T1}
\ddot{Y}(t)+ F'(Y(t))=0\quad,\quad (Y(0),\dot{Y}(0))=(Y_0,0)
\end{equation}
where $Y_0>0$. Then equation \eqref{T1} has a periodic non constant
solution if and only if the function $G : z\longmapsto
2(F(Y_0)-F(z))$ has two simples distinct zeros $\alpha$ and $\beta$
with $\alpha\leq Y_0\leq \beta$ and such that $G$ has no zero in the
interval $]\alpha,\beta[$. In this case, the period is given by
\begin{equation}\label{T2}
T=2\int_\alpha^\beta\frac{dz}{\sqrt{G(z)}}\;=2\int_\alpha^\beta\frac{dz}{\sqrt{2(F(Y_0)-F(z))}}.
\end{equation}
\end{lem}
Taking $F(z)=\frac{e^{4\pi z^2}-1-4\pi z^2}{8\pi}$ in the above Lemma, the
solution $\phi_k$ is periodic and we have
\begin{eqnarray*}
T_k&=&2\int_{-(1+\frac 1 k)\sqrt\frac{k}{4\pi}}^{(1+\frac 1 k)\sqrt\frac{k}{4\pi}}\frac{dz}{\sqrt{2[(e^{(1+\frac 1 k)^2k}-(1+\frac 1 k)^2k)-(e^{4\pi z^2}-4\pi z^2)]}}\\
&=&4\int_0^{(1+\frac 1 k)\sqrt k} \frac{du}{\sqrt{(e^{(1+\frac 1 k)^2k}-(1+\frac 1 k)^2k)-(e^{u^2}-u^2)}}.
\end{eqnarray*}
Now to estimate the period $T_k$ we use the following Lemma.
\begin{lem}\label{T3}
For any $A>1$, we have
$$
\int_0^A\frac{du}{\sqrt{(e^{A^2}-A^2)-(e^{u^2}-u^2)}}\leq \sqrt{1-2e^{-1}}e^{\frac{-A^2}{2}}[A-\frac{1}{A}+\frac{A}{A^2-1}].
$$

\end{lem}
\begin{proof} Write,
$$
\int_0^A\frac{du}{\sqrt{(e^{A^2}-A^2)-(e^{u^2}-u^2)}}=\int_0^{A-\frac{1}{A}}+\int_{A-\frac{1}{A}}^A
$$
The first term on the right-hand side can be estimated by
$\sqrt{1-2e^{-1}}(A-\frac{1}{A})e^{\frac{-A^2}{2}}.$ Let
$h(u)=\frac{1}{u(e^{u^2}-1)}$ and
$g'(u)=\frac{u(e^{u^2}-1)}{\sqrt{(e^{A^2}-A^2)-(e^{u^2}-u^2)}}.$ Integrating by
parts in the second integral, we obtain

\begin{eqnarray}\label{ped}
\int_{A-\frac{1}{A}}^A\frac{du}{\sqrt{(e^{A^2}-A^2)-(e^{u^2}-u^2)}}&\leq&
\frac{A}{A^2-1}\sqrt{e^{A^2}-e^{A^2-2-\frac1{A^2}}-2+\frac1{A^2}}\\
&\leq&\sqrt{1-2e^{-1}}\frac{A}{A^2-1}e^{\frac{-A^2}{2}},
\nonumber
\end{eqnarray}
giving,

$$
\int_0^A\frac{du}{\sqrt{(e^{A^2}-A^2)-(e^{u^2}-u^2)}}\leq \sqrt{1-2e^{-1}}e^{\frac{-A^2}{2}}[A-\frac{1}{A}+\frac{A}{A^2-1}]
$$
as desired.
\end{proof}

Choosing $A=\sqrt{k}(1+\frac{1}{k})$ in the above Lemma \ref{T3} with $k$ large enough, we get

\begin{eqnarray*}
T_k&\leq&
\sqrt{k}e^{\frac{-k}{2}(1+\frac{1}{k})^2}4e^2\left[\frac{(k+1)^2-k}{k(k+1)}+\frac{(k+1)}{(k+1)^2-k}\right]\\
&\leq&C_1\sqrt{k}e^{\frac{-k}{2}(1+\frac{1}{k})^2}.
\end{eqnarray*}
Since $\phi_k$ is a periodic function and decreasing on $]0,T_k/4[$
(actually, $\phi_k$ may be viewed as a cosine function) then, we choose
$t_k\in]0,T_k/4[$ such that
$$
\phi_k(t_k)=(1+\frac{1}{k})\sqrt{\frac{k}{4\pi}}-\left((1+\frac{1}{k})\sqrt{\frac{k}{4\pi}}\right)^{-1}.
$$
Clearly,

$$
t_k=\int_{\sqrt k+\frac{1}{\sqrt k}-\frac{4\pi\sqrt k}{k+1}}^{\sqrt
k+\frac{1}{\sqrt
k}}\frac{du}{\sqrt{(e^{(1+\frac 1 k)^2k}-(1+\frac 1 k)^2k)-(e^{u^2}-u^2)}}.
$$
Using \eqref{ped} with $A=\sqrt{k}+\frac{1}{\sqrt{k}}$, we obtain

\begin{eqnarray*}
t_k&\leq&e^{8\pi}\frac{k(k+1)}{\sqrt{k}(k^2+(2-4\pi)k+1)}e^{-\frac{1}{2}(\sqrt{k}+\frac{1}{\sqrt
k})^2}
\\
&\leq&e^{8\pi}\frac{e^{-k/2}}{\sqrt k}\frac{k(k+1)}{(k^2+(2-4\pi)k+1)}.\\
\end{eqnarray*}
Then, if $k$ is large enough
$$
t_k\leq \frac \eta 2 e^{-k/2}.
$$
Finally, we will prove that this time $t_k$ is sufficient to establish the instability result.\\
Since,
\begin{eqnarray*}
\|\p_t(v_k-w_k)(t_k)\|^2_{L^2(\Omega)}&=&\int_{\Omega}|\p_t(v_k-w_k)(t_k)|^2\;dx\\
&\geq&
\int_{|x|<\frac{\eta}{2}e^{-k/2}}|\p_t(v_k-w_k)(t_k)|^2\;dx=\frac{\pi}{4}\eta^2e^{-k}|\p_t(\phi_k-\psi_k)(t_k)|^2.
\end{eqnarray*}
Then, it suffices to estimate $|\p_t(\phi_k-\psi_k)(t_k)|.$ To do
so, we can write
$$
|\p_t(\phi_k-\psi_k)(t_k)|=\displaystyle\frac{|(\p_t\phi_k(t_k))^2-(\p_t\psi_k(t_k))^2|}{|\p_t\phi_k(t_k)+\p_t\psi_k(t_k)|},
$$
with
\begin{equation}
\p_t\phi_k(t_k)^2=\frac{e^{4\pi \phi_k(0)^2}-4\pi \phi_k(0)^2-e^{4\pi
\phi_k(t_k)^2}+4\pi \phi_k(t_k)^2       }{4\pi}
\end{equation}
and similarly
\begin{equation}
\p_t\psi_k(t_k)^2=\frac{e^{4\pi \psi_k(0)^2}-4\pi \psi_k(0)^2-e^{4\pi
\psi_k(t_k)^2}+4\pi \psi_k(t_k)^2}{4\pi}.
\end{equation}
Hence,
\begin{eqnarray*}
|\p_t\phi_k(t_k))^2-(\p_t\psi_k(t_k))^2|
=\left|\frac{e^{4\pi
\phi_k(0)^2}-e^{4\pi \phi_k(t_k)^2}-e^{4\pi \psi_k(0)^2}+e^{4\pi
\psi_k(t_k)^2}}{4\pi}\right|.
\end{eqnarray*}
Using the fact that $\psi_k$ is decreasing on $[0,T_k/4]$, we have

$$
|e^{4\pi \psi_k(0)^2}-4\pi \psi_k(0)^2-e^{4\pi
\psi_k(t_k)^2}+4\pi \psi_k(t_k)^2|
\leq 2e^k.
$$
In addition,
$$
e^{4\pi\phi_k(0)^2}-e^{4\pi \phi_k(t_k)^2}=e^{k+\frac 1 k+2}-e^{2-8\pi+k+\frac 1 k +\frac {16\pi^2 k}{(k+1)^2}}.
$$
Therefore for $k$ large enough,

\begin{eqnarray*}
|(\p_t\phi_k(t_k))^2-(\p_t\psi_k(t_k))^2|
\geq Ce^{k}.
\end{eqnarray*}
Moreover,
\begin{eqnarray*}
|\p_t\phi_k(t_k)+\p_t\psi_k(t_k)|&\leq&\frac{e^{2\pi\phi_k(0)^2}+e^{2\pi\psi_k(0)^2}}{\sqrt{4\pi}}\\
&\leq&\frac{e^{k/2}+e^{\frac{k}{2}(1+\frac{1}{k})^2}}{\sqrt{4\pi}}\\
&\leq&\frac{e^{k/2}(e^{1+\frac{1}{k^2}+\frac 2 k}+1)}{\sqrt{4\pi}}.
\end{eqnarray*}
For large  $k$, we have
$$
|\p_t\phi_k(t_k)+\p_t\psi_k(t_k)|\leq Ce^{k/2},
$$
and consequently,
$$
|\p_t(\phi_k-\psi_k)(t_k)|^2\geq C e^k.
$$
Finally, we obtain
\begin{equation}
\liminf_{k\rightarrow\infty}\|\p_t(v_k-w_k)(t_k)\|^2_{L^2(\Omega)}\geq\frac
\pi 4 C\eta^2.
\end{equation}

This finishes the proof of Theorem \ref{ins}.
\endproof

\setcounter{equation}{0} \setcounter{equation}{0}

\end{document}